\documentclass{amsart}%
\usepackage{amsmath}
\usepackage{amsfonts}
\usepackage{amssymb}
\usepackage{graphicx}%
\usepackage{palatino,mathpazo}
\providecommand{\U}[1]{\protect\rule{.1in}{.1in}}
\newtheorem{theorem}{Theorem}

\newtheorem{corollary}[theorem]{Corollary}

\newtheorem{lemma}[theorem]{Lemma}

\newtheorem{proposition}[theorem]{Proposition}
\newtheorem{remark}[theorem]{Remark}

\begin{document}
	\author[Yong Liu]{Yong Liu}
	\address{Department of Mathematics, University of Science and Technology of China,
		Hefei, P.R. China}
	\email{yliumath@ustc.edu.cn}
	\author[Juncheng Wei]{Juncheng Wei}
	\address{Department of Mathematics, University of British Columbia, Vancouver, B.C.
		Canada, V6T 1Z2.}
	\email{jcwei@math.ubc.ca}
	\author[Wen Yang]{Wen Yang}
	\address{Department of Mathematics, Faculty of Science and Technology, University of Macau, Macau, P.R. China}
	\email{wenyang@um.edu.mo}
	\title{Lump type solutions: Backlund transformation and spectral properties}
	\maketitle
	
	\begin{abstract}
		There are various different ways to obtain traveling waves of lump type for
		the KP equation. We propose a general and simple approach to derive them via a
		Backlund transformation. This enables us to establish an explicit connection
		between those low energy solutions and high energy ones. Based on this
		construction, spectral analysis of the degree $6$ solutions is then carried
		out in details. The analysis of higher energy ones can be done in an inductive way.
		
	\end{abstract}

	\section{\bigskip Introduction and statement of main results}
	
	KP equation is a two dimensional analogy of the classical KdV equation. It
	naturally appears in the theory of shallow water waves. As an important $2+1$
	dimensional integrable system, it has been extensively studied for more than
	forty years, and many other integrable systems can be regarded as its suitable
	reduction. However, there are still some interesting questions remained to be
	answered for this equation. As a matter of fact, even the properties of its
	traveling wave solutions are not fully understood yet. For traveling waves,
	the KP equation reduces to the so called Boussinesq equation:
	\begin{equation}
		\partial_{x}^{2}\left(  \partial_{x}^{2}u-u+3u^{2}\right)  -\partial_{y}%
		^{2}u=0. \label{Bou}%
	\end{equation}
	In principle, solutions to this equation should play important role in the
	long time dynamics of the KP equation.
	
	\medskip
	
	The Boussinesq equation in the form of $\left(  \ref{Bou}\right)  $ is of
	elliptic type and closely related to other PDEs such as GP equation. While
	there already exist a lot of works concerning elliptic equations of second
	order, the study of fourth order equations with both mathematical and physical
	significance is relatively few. In this paper, we would like to study the
	spectral properties of \textquotedblleft lump type\textquotedblright%
	\ solutions to $\left(  \ref{Bou}\right)  $. By lump type solutions, we mean
	solutions of $\left(  \ref{Bou}\right)  $ which decay to zero at infinity.
	This is a natural class of physically meaningful solutions. The name
	\textquotedblleft lump type\textquotedblright\ comes from the fact that the
	following \textquotedblleft classical\textquotedblright\ lump solution solves
	$\left(  \ref{Bou}\right)  $:
	\[
	U\left(  x,y\right)  =4\frac{y^{2}-x^{2}+3}{\left(  x^{2}+y^{2}+3\right)
		^{2}}.
	\]
	The analysis of $U$ has a long history. It is first obtained in
	\cite{Manakovz,Sat} by parameter degeneration. The spectral property of $U$ is
	now well understood. Indeed, we have proved in \cite{Liu} using Backlund
	transformation that $U$ is nondegenerated, in the sense that the linearized
	operator at $U$ does not have any nontrivial kernels. A direct consequence of
	this property is that the lump is orbitally stable under the KP-I flow, which
	is globally well posed \cite{Kenig,Moli,Moli2}. The asymptotical stability of
	$U$ remains to be an unsolved open problem in this direction.
	
	\medskip
	
	If we introduce the tau function $\tau$ by $u=2\partial_{x}^{2}\ln\tau,$ then
	the equation $\left(  \ref{Bou}\right)  $ turns into the following bilinear
	equation:
	\begin{equation}
		\left(  D_{x}^{4}-D_{x}^{2}-D_{y}^{2}\right)  \tau\cdot\tau=0.
		\label{bilinear}%
	\end{equation}
	Here $D$ is the bilinear derivative operator introduced by Hirota
	\cite{Hirota}. One easily checks that the lump solution $U$ corresponds to
	$\tau=x^{2}+y^{2}+3.$
	
	\medskip
	
	Our recent result \cite{LYW} shows that real valued solution(with a mild
	decaying assumption) of $\left(  \ref{Bou}\right)  $ has to be rational, and
	the corresponding tau function, which solves $\left(  \ref{bilinear}\right)
	,$ will be a polynomial of degree $k\left(  k+1\right)  $ with $k\in
	\mathbb{N}$. In the case of $k=1,$ it is not difficult to show that up to
	translation in the $x,y$ variables and multiplication by a constant, real
	valued solution to $\left(  \ref{bilinear}\right)  $ has to be $x^{2}%
	+y^{2}+3.$
	
	\medskip
	
	According to our classification result mentioned above, the next family of tau
	functions of $\left(  \ref{bilinear}\right)  $ are polynomials of degree $6,$
	corresponding to $k=2.$ In Section 2, we show, using Backlund transformation,
	that the following family of polynomials $h_{A,B}$ solves $\left(
	\ref{bilinear}\right)  ,$ where
	\begin{align*}
		h_{A,B}\left(  x,y\right)    =~& x^{6}+3x^{4}y^{2}+3x^{2}y^{4}+y^{6}%
		+25x^{4}+90x^{2}y^{2}+17y^{4}\\
		&  +Bx^{3}+3Ax^{2}y-3Bxy^{2}-Ay^{3}-125x^{2}+475y^{2}\\
		&  -Bx+5Ay+1875+\frac{A^{2}}{4}+\frac{B^{2}}{4}.
	\end{align*}
	Here $A,B\in\mathbb{R}$ are parameters. The solutions $2\partial_{x}^{2}\ln
	h_{A,B}$ of $\left(  \ref{Bou}\right)  $ will be denoted by $u_{A,B}.$
	Presumably, any degree $6$ solution should belong to this family. Note that
	this family of solutions are first obtained in \cite{Pelinovskii93}, using a
	limiting procedure on the involved parameters.
	
	\medskip
	
	As we know, solutions of most elliptic equations do not have explicit
	expression. Therefore the existence of such a family of
	non-radially-symmetric, rational, solutions $u_{A,B}$ to the Boussinesq
	equation is by itself an interesting phenomenon and provides an important
	scenario for us to analyze various properties of the solutions relevant in the
	subject of elliptic equations.
	
	Our first result is the following
	
	\begin{theorem}
		\label{main}For any $A,B,$ the solution $u=u_{A,B}$ is nondegenerated in the
		following sense: If $\phi$ is a solution of the linearized operator
		\[
		\partial_{x}^{2}\left(  \partial_{x}^{2}\phi-\phi+6u\phi\right)  -\partial
		_{y}^{2}\phi=0,
		\]
		with
		\[
		\phi\left(  x,y\right)  \rightarrow0\text{, as }x^{2}+y^{2}\rightarrow
		+\infty,
		\]
		then there exist constants $c_{1},...,c_{4},$ such that
		\[
		\phi=c_{1}\partial_{x}u+c_{2}\partial_{y}u+c_{3}\partial_{A}u+c_{4}%
		\partial_{B}u.
		\]
		Moreover, the Morse index of $u$ is equal to $3.$
	\end{theorem}
	
	The Morse index of $u$ is, by definition, the number of negative
	eigenvalues(counted by multiplicity) of the operator
	\[
	\eta\rightarrow\partial_{x}^{2}\phi-\phi+6u\phi-\partial_{x}^{-2}\partial
	_{y}^{2}\phi.
	\]

	\medskip
	
	In the special case that $A=B=0,$ the solution $u_{0,0}$ will be even in both
	variables. This is precisely the solution obtained for the first time in
	\cite{Pelinovskii}. If we consider kernel which is also even in both
	variables, then from Theorem \ref{main}, we know that it has to be $0.$ As a
	direct application of this fact, one can then construct new subsonic traveling
	wave solutions to the GP equation from $u_{0,0}$, using the same perturbation
	method developed in \cite{LWW}. More precisely, we have
	
	\begin{corollary}
		For each $\varepsilon>0$ sufficiently small, there exists a traveling wave
		solution
		\[
		\Phi\left(  t,x,y\right)  =\phi\left(  x-\left(  \sqrt{2}-\varepsilon
		^{2}\right)  t,y\right)  ,
		\]
		to the following GP equation%
		\[
		i\partial_{t}\Phi+\Delta\Phi+\left(  1-\left\vert \Phi\right\vert ^{2}\right)
		\Phi=0,
		\]
		which has the asymptotic expansion%
		\[
		\phi\left(  x,y\right)  =1+i\varepsilon\partial_{x}^{-1}u^{\ast}+O\left(
		\varepsilon^{2}\right)  ,
		\]
		where
		\[
		u^{\ast}\left(  x,y\right)  =-\frac{3}{2}u_{0,0}\left(  2^{\frac{3}{4}%
		}x,2^{\frac{1}{2}}y\right)  .
		\]
		
	\end{corollary}
	
	To the best of our knowledge, this is the first result in the construction of
	high energy traveling wave solution for the GP equation in the subsonic
	regime, although there are quite a few existence results on the least energy solutions.
	
	\medskip
	
	From the even solution $u_{0,0}$ and our nondegeneracy result, one can
	actually also construct new nontrivial solutions to following generalized KP-I
	equation:%
	\[
	\partial_{x}^{2}u-u+\left\vert u\right\vert ^{\alpha}-\partial_{x}%
	^{-2}\partial_{y}^{2}u=0,
	\]
	provided that the exponent $\alpha$ is sufficiently close to $2.$ Note that
	for this type of generalized KP-I equation, one can use variational method
	(See \cite{Saut97,Bouard97,LiuYue}) to construct the ground state solution,
	whose Morse index is presumably equal to $1.$ On the other hand, the new
	solutions close to $u_{0,0}$ will have Morse index $3,$ and it seems to be
	hard to construct them by variational method. It is also interesting to see
	whether this type of solutions exist for $\alpha$ not close to $2.$
	
	\medskip
	
	The method developed in this paper to prove Theorem \ref{main} can actually be
	used to construct and analyze the spectral properties of higher energy lump
	type solutions. We emphasize that general lump type solutions of KP equation
	with degree $k\left(  k+1\right)  ,$ and $k$ free parameters have already been
	found in \cite{Pelinovskii95,Pelinovskii,Pelnovsky94,Pelnovsky98}, using the
	Wronskian representation. It is also worth mentioning that there are also
	other methods to construct these solutions, see
	\cite{Ablowitz,Airault,Clarkson,Prada,Vi,Yang}. Our method has the advantage
	that it establishes explicit connection between low energy solutions and high
	energy ones. The complete classification of lump type solutions remains open.
	
	\medskip
	
	The paper is organized in the following way. In Section 2, we explain how to
	use the Backlund transformation to create higher energy solutions from the low
	energy solutions. We point out that these solutions are in general complex
	valued. In Section 3, we analyze the precise asymptotic behavior of the
	eigenfunctions of the linearized operator and show that the Morse index of
	$h_{A,B}$ is equal to $3.$ This is based on a \textquotedblleft
	reverse\textquotedblright\ Lyapunov-Schmidt reduction type argument. The very
	delicate point here is that the reduced problem is actually degenerated.
	\medskip

	\noindent \textbf{Acknowledgement}\quad  Y. Liu was supported by the National Key R\&D Program of China 2022YFA1005400 and NSFC 11971026, NSFC 12141105.  J.Wei was supported by NSERC of Canada.  W. Yang was supported by National Key R\&D Program of China 2022YFA1006800, NSFC No.12171456 and NSFC No.12271369.
	
	\section{Backlund transformation from low energy solutions to high energy
		ones}
	
	In this section, we propose a general scheme to create high energy lump type
	solutions from low energy ones. Although there are other methods to construct
	these solutions, our method has the advantage that it establishes explicit
	link between solutions with different energy, which in turn enables us to
	study their spectral properties using an inductive argument.

	\medskip
	
	The Boussinesq equation has the following Backlund transformation(See
	\cite{Liu}):
	\begin{equation}
		\label{Back}
		\begin{cases} 
			\left(  D_{x}^{2}+\mu D_{x}+\frac{i}{\sqrt{3}}D_{y}-\lambda\right)  f\cdot
			g=0,\\
			\\
			\left(  \left(  3\lambda-1\right)  D_{x}-\sqrt{3}i\mu D_{y}+D_{x}^{3}-\sqrt
			{3}iD_{x}D_{y}+v\right)  f\cdot g=0.
		\end{cases}
	\end{equation}
	Here $\mu,\lambda,v$ are arbitrary parameters, and throughout the paper, $i$
	will be the imaginary unit. The Backlund transformation $\left(
	\ref{Back}\right)  $ has the following property: If $f$ satisfies the bilinear
	equation $\left(  \ref{bilinear}\right)  ,$ and $f,g$ satisfy the system
	$\left(  \ref{Back}\right)  ,$ then $g$ will automatically be a solution of
	$\left(  \ref{bilinear}\right)  .$ In this way, to get a new solution, we only
	need to solve a linear problem involving third order derivatives, instead of
	the original nonlinear problem with fourth order derivatives.
	
	\medskip
	
	If $f,g$ are polynomials, then in view of the highest degree terms,
	necessarily $\lambda=v=0.$ Then inspecting the highest degree terms, we find
	that $\mu=\pm\frac{1}{\sqrt{3}}$. We are thus lead to consider the following
	\begin{equation}
		\label{Back2}
		\begin{cases}
			\left(  D_{x}^{2}+\mu D_{x}+\frac{i}{\sqrt{3}}D_{y}\right)  f\cdot g=0,\\
			\\
			\left(  -D_{x}-\sqrt{3}i\mu D_{y}+D_{x}^{3}-\sqrt{3}iD_{x}D_{y}\right)  f\cdot
			g=0.
		\end{cases}
	\end{equation}

	Let $f$ be a real valued polynomial solution of $\left(  \ref{bilinear}%
	\right)  $ of degree $2n.$ The classification result in \cite{LYW} tells us
	that
	\[
	n=\frac{k\left(  k+1\right)  }{2},\text{ for some }k\in\mathbb{N}.
	\]
	Those degree $j$ terms in $f$ will be denoted by $f_{j}.$ Our first key
	observation is that it will be more convenient to consider the problem in the
	$z$-$\bar{z}$ coordinate, rather than the usual $x$-$y$ coordinate, where%
	\[
	z=x+yi\text{ and }\bar{z}=x-yi.
	\]
	It follows from Lemma 13 of \cite{LYW} that $f_{2n}=z^{n}\bar{z}^{n}$, and by
	suitable translation in the $x$ and $y$ variables, we can assume that
	$f_{2n-1}=0.$ We also observe that
	\[
	D_{x}+iD_{y}=2D_{\bar{z}}\text{ \ and }D_{x}-iD_{y}=2D_{z}.
	\]
	This formula, although very simple, will be frequently used in this section.

	\medskip
	
	The following lemma tells us that if $f,g$ are connected through Backlund
	transformation, then the degree of $g$ will be a square number and essentially
	determined by that of $f,$ and more importantly, $g$ will be complex valued.
	
	\begin{lemma}
		\label{degr}Let $f$ be a real valued polynomial solution of $\left(
		\ref{bilinear}\right)  $ of degree $2n$ with $f_{2n-1}=0.$ Suppose $f,g$
		satisfies $\left(  \ref{Back2}\right)  .$ Then the highest degree term $g_{m}$
		of $g$ has the form
		\[
		g_{m}=z^{j}\bar{z}^{n},
		\]
		where $j$ satisfies
		\begin{equation}
			n\left(  n-1\right)  -2nj+j\left(  j-1\right)  =0. \label{jn}%
		\end{equation}
		In particular, if $n=\frac{k\left(  k+1\right)  }{2}$ for some integer $k,$
		then the degree $j+n$ of $g$ is equal to
		\[
		k^{2}\text{ or }\left(  k+1\right)  ^{2}.
		\]
		Moreover, $g_{m-1}=\sqrt{3}nz^{j}\bar{z}^{n-1}+cz^{j-1}\bar{z}^{n},$ where
		\[
		c=-\frac{\left(  j-n\right)  \left(  j+n-1\right)  }{2\sqrt{3}\left(
			n-j+1\right)  },
		\]
		and $g_{m-2}$ solves
		\[
		2D_{\bar{z}}f_{2n}\cdot g_{m-2}+2D_{\bar{z}}f_{2n-2}\cdot g_{m}+\sqrt{3}%
		D_{x}^{2}f_{2n}\cdot g_{m-1}=0.
		\]
		
	\end{lemma}
	
	\begin{proof}
		Balancing highest degree terms in $\left(  \ref{Back2}\right)  $ requires
		\[
		D_{\bar{z}}f_{2n}\cdot g_{m}=0.
		\]
		This readily implies
		\[
		g_{m}=cz^{j}\bar{z}^{n},
		\]
		for some constant $c\ $and non-negative integer $j.$ The constant $c$ can be
		normalized to be $1.$
		
		Since $f_{2n-1}=0,$ the degree $m-1$ terms in $g$ should satisfy
		\begin{equation}
			\label{gm-1}
			\begin{cases}
				2D_{\bar{z}}f_{2n}\cdot g_{m-1}+\sqrt{3}D_{x}^{2}f_{2n}\cdot g_{m}=0,\\
				\\
				2D_{\bar{z}}f_{2n}\cdot g_{m-1}+\sqrt{3}iD_{x}D_{y}f_{2n}\cdot g_{m}=0.
			\end{cases}
		\end{equation}
		We compute
		\begin{align*}
			D_{x}^{2}\left(  z^{n}\bar{z}^{n}\right)  \cdot\left(  z^{j}\bar{z}%
			^{k}\right)     =~&\left[  n\left(  n-1\right)  -2nj+j\left(  j-1\right)
			\right]  z^{j+n-2}\bar{z}^{n+k}\\
			&  +\left[  2n^{2}-2nk-2nj+2jk\right]  z^{j+n-1}\bar{z}^{n+k-1}\\
			&  +\left[  n\left(  n-1\right)  -2nk+k\left(  k-1\right)  \right]
			z^{j+n}\bar{z}^{n+k-2}.
		\end{align*}
		Similarly, we have
		\begin{align*}
			D_{x}D_{y}\left(  z^{n}\bar{z}^{n}\right)  \cdot\left(  z^{j}\bar{z}%
			^{k}\right)     =~&\left[  n\left(  n-1\right)  -2nj+j\left(  j-1\right)
			\right]  iz^{j+n-2}\bar{z}^{k+n}\\
			&  +\left[  -n\left(  n-1\right)  +2nk-k\left(  k-1\right)  \right]
			iz^{j+n}\bar{z}^{k+n-2}.
		\end{align*}
		From the system $\left(  \ref{gm-1}\right)  ,$ we deduce
		\[
		D_{x}^{2}f_{2n}\cdot g_{m}=iD_{x}D_{y}f_{2n}\cdot g_{m},
		\]
		which implies that when $k=n,$ there holds
		\begin{align*}
			&  \left[  n\left(  n-1\right)  -2nj+j\left(  j-1\right)  \right]
			z^{j+n-2}\bar{z}^{n+k}+\left[  n\left(  n-1\right)  -2nk+k\left(  k-1\right)
			\right]  z^{j+n}\bar{z}^{n+k-2}\\
			&  =-\left[  n\left(  n-1\right)  -2nj+j\left(  j-1\right)  \right]
			z^{j+n-2}\bar{z}^{k+n}\\
			&\quad-\left[  -n\left(  n-1\right)  +2nk-k\left(  k-1\right)
			\right]  z^{j+n}\bar{z}^{k+n-2}.
		\end{align*}
		From this we get the relation $\left(  \ref{jn}\right)  $ between $j$ and $n.$
		
		Now $g_{m-1}$ satisfies%
		\[
		D_{\bar{z}}f_{2n}\cdot g_{m-1}=-\frac{\sqrt{3}}{2}D_{x}^{2}f_{2n}\cdot
		g_{m}=\sqrt{3}nz^{j+n}\bar{z}^{2n-2}.
		\]
		Therefore we obtain
		\[
		g_{m-1}=\sqrt{3}nz^{j}\bar{z}^{n-1}+cz^{j-1}\bar{z}^{n},
		\]
		where $c$ is a constant to be determined.
		
		To find $c,$ we consider the equations to be satisfied by $g_{m-2}:$
		\[
		\begin{cases}
			2D_{\bar{z}}f_{2n}\cdot g_{m-2}+2D_{\bar{z}}f_{2n-2}\cdot g_{m}+\sqrt{3}%
			D_{x}^{2}f_{2n}\cdot g_{m-1}=0,\\
			\\
			2D_{\bar{z}}f_{2n}\cdot g_{m-2}+2D_{\bar{z}}f_{2n-2}\cdot g_{m}+\sqrt{3}%
			iD_{x}D_{y}f_{2n}\cdot g_{m-1}-D_{x}^{3}f_{2n}\cdot g_{m}=0.
		\end{cases}
		\]
		We compute
		\begin{align*}
			&  D_{x}^{3}\left(  z^{n}\bar{z}^{n}\right)  \cdot\left(  z^{j}\bar{z}%
			^{k}\right) \\
			&  =\left[  n\left(  n-1\right)  \left(  n-2\right)  -3kn\left(  n-1\right)
			+3nk\left(  k-1\right)  -k\left(  k-1\right)  \left(  k-2\right)  \right]
			z^{j+n}\bar{z}^{k+n-3}\\
			& \quad+\left[  3n^{2}\left(  n-1\right)  -3jn\left(  n-1\right)  -6n^{2}%
			k+6njk+3nk\left(  k-1\right)  -3jk\left(  k-1\right)  \right]  z^{j+n-1}%
			\bar{z}^{k+n-2}\\
			&\quad  +\left[  3n^{2}\left(  n-1\right)  -6n^{2}j-3nk\left(  n-1\right)
			+6njk+3nj\left(  j-1\right)  -3j\left(  j-1\right)  k\right]  z^{j+n-2}\bar
			{z}^{k+n-1}\\
			&\quad  +\left[  n\left(  n-1\right)  \left(  n-2\right)  -3nj\left(  n-1\right)
			+3nj\left(  j-1\right)  -j\left(  j-1\right)  \left(  j-2\right)  \right]
			z^{j+n-3}\bar{z}^{k+n}.
		\end{align*}
		We find that if $j$ and $n$ satisfy $\left(  \ref{jn}\right)  ,$ then
		\[
		D_{x}^{3}\left(  z^{n}\bar{z}^{n}\right)  \cdot\left(  z^{j}\bar{z}%
		^{n}\right)  =\left(  6jn-6n^{2}\right)  z^{j+n-1}\bar{z}^{2n-2}+2\left(
		j-n\right)  \left(  j+n-1\right)  z^{j+n-3}\bar{z}^{2n},
		\]
		and moreover,
		\begin{align*}
			&  \sqrt{3}D_{x}^{2}f_{2n}\cdot g_{m-1}-\sqrt{3}iD_{x}D_{y}f_{2n}\cdot
			g_{m-1}+D_{x}^{3}f_{2n}\cdot g_{m}\\
			&  =2\sqrt{3}D_{x}D_{z}\left(  z^{n}\bar{z}^{n}\right)  \left(  \sqrt{3}%
			nz^{j}\bar{z}^{n-1}+cz^{j-1}\bar{z}^{n}\right)  +D_{x}^{3}\left(  z^{n}\bar
			{z}^{n}\right)  \left(  z^{j}\bar{z}^{n}\right) \\
			&  =2\sqrt{3}c\left(  2n-2j+2\right)  z^{j+n-3}\bar{z}^{2n}+2\left(
			j-n\right)  \left(  j+n-1\right)  z^{j+n-3}\bar{z}^{k+n}.
		\end{align*}
		Compatibility of the two equations in $\left(  \ref{Back2}\right)  $ requires
		the right hand side to be $0.$ It follows that
		\[
		c=-\frac{\left(  j-n\right)  \left(  j+n-1\right)  }{2\sqrt{3}\left(
			n-j+1\right)  }.
		\]
		This proves the assertion.
	\end{proof}
	
	This lemma tells us that if $n=1,$ then the degree of $g$ has to be $1$ or
	$4.$ We now proceed to construct explicit Backlund transformations from
	\begin{equation}
		f=x^{2}+y^{2}+3, \label{deg2}%
	\end{equation}
	the first nontrivial solution, to a family of degree $4$ polynomial. We will
	see that there will be a free complex parameter appearing in the process.
	
	\begin{proposition}
		Let $f\ $be given by $\left(  \ref{deg2}\right)  $ and
		\[
		g=z^{3}\bar{z}+\sqrt{3}z^{3}+\sqrt{3}z^{2}\bar{z}+12z^{2}-3\bar{z}^{2}%
		+3z\bar{z}+9\sqrt{3}z+\alpha\bar{z}-36+\sqrt{3}\alpha,
		\]
		where $\alpha\in\mathbb{C}$ is a parameter. Then $f,g$ satisfies
		\[
		\begin{cases}
			\left(  D_{x}^{2}+\frac{1}{\sqrt{3}}D_{x}+\frac{i}{\sqrt{3}}D_{y}\right)
			f\cdot g=0,\\
			\\
			\left(  D_{x}^{3}-\sqrt{3}iD_{x}D_{y}-D_{x}-iD_{y}\right)  f\cdot g=0.
		\end{cases}
		\]
		As a consequence, $g$ is a solution to the bilinear equation $\left(
		\ref{bilinear}\right)  $.
	\end{proposition}
	
	\begin{proof}
		Note that $f_{2}=z\bar{z}$ and $f_{0}=3.$ Applying Lemma \ref{degr}, we find
		$g_{4}=z^{3}\bar{z}$ and $g_{3}=\sqrt{3}z^{3}+\sqrt{3}z^{2}\bar{z}.$
		
		To get $g_{2},$ we observe that it satisfies
		\[
		2D_{\bar{z}}f_{2}\cdot g_{2}+2D_{\bar{z}}f_{0}\cdot g_{4}+\sqrt{3}D_{x}%
		^{2}f_{2}\cdot g_{3}=0.
		\]
		Solving this equation gives
		\[
		g_{2}=12z^{2}-3\bar{z}^{2}+az\bar{z},
		\]
		where $a$ is a parameter still to be determined in the next step.
		
		To find $a$ and $g_{1},$ we use the fact that $g_{1}$ satisfies
		\[
		\begin{cases}
			2D_{\bar{z}}f_{2}\cdot g_{1}+2D_{\bar{z}}f_{0}\cdot g_{3}+\sqrt{3}D_{x}%
			^{2}f_{2}\cdot g_{2}+\sqrt{3}D_{x}^{2}f_{0}\cdot g_{4}=0,\\
			\\
			2D_{\bar{z}}f_{2}\cdot g_{1}+2D_{\bar{z}}f_{0}\cdot g_{3}+\sqrt{3}iD_{x}%
			D_{y}f_{2}\cdot g_{2}+\sqrt{3}iD_{x}D_{y}f_{0}\cdot g_{4}-D_{x}^{3}f_{2}\cdot
			g_{3}=0.
		\end{cases}
		\]
		Compatibility of these two equations implies that $a=3$ and
		\[
		g_{2}=12z^{2}-3\bar{z}^{2}+3z\bar{z}.
		\]
		Now solving the following equation for $g_{1}:$%
		\[
		2D_{\bar{z}}f_{2}\cdot g_{1}+2D_{\bar{z}}f_{0}\cdot g_{3}+\sqrt{3}D_{x}%
		^{2}f_{2}\cdot g_{2}+\sqrt{3}D_{x}^{2}f_{0}\cdot g_{4}=0,
		\]
		we get
		\[
		g_{1}=9\sqrt{3}z+\alpha\bar{z},
		\]
		where $\alpha$ is a parameter.
		
		To see whether or not $\alpha$ can be arbitrary, we consider the equations
		satisfied by $g_{0}:$
		\[
		\begin{cases}
			2D_{\bar{z}}f_{2}\cdot g_{0}+2D_{\bar{z}}f_{0}\cdot g_{2}+\sqrt{3}D_{x}%
			^{2}f_{2}\cdot g_{1}+\sqrt{3}D_{x}^{2}f_{0}\cdot g_{3}=0,\\
			\\
			2D_{\bar{z}}f_{2}\cdot g_{0}+2D_{\bar{z}}f_{0}\cdot g_{2}+\sqrt{3}iD_{x}%
			D_{y}f_{2}\cdot g_{1}+\sqrt{3}iD_{x}D_{y}f_{0}\cdot g_{3}-D_{x}^{3}f_{2}\cdot
			g_{2}-D_{x}^{3}f_{0}\cdot g_{4}=0.
		\end{cases}
		\]
		Direct computation tells us that this system is compatible for any $\alpha,$
		and we are lead to
		\[
		2D_{\bar{z}}f_{2}\cdot g_{0}+\left(  72-2\sqrt{3}\alpha\right)  z=0,
		\]
		which implies that $g_{0}=-36+\sqrt{3}\alpha.$ One then checks that the
		function $g$ obtained in this way indeed solves $\left(  \ref{Back2}\right)
		.$ This finishes the proof.
	\end{proof}
	
	\subsection{A family of degree $6$ tau functions and their Backlund
		transformation}
	
	We have obtained a family of degree $4$ polynomials:%
	\[
	g=z^{3}\bar{z}+\sqrt{3}z^{3}+\sqrt{3}z^{2}\bar{z}+12z^{2}-3\bar{z}^{2}%
	+3z\bar{z}+9\sqrt{3}z+\alpha\bar{z}-36+\sqrt{3}\alpha.
	\]
	We would like to find all the degree $6$ polynomial $h$ such that $g,h$ are
	connected through the Backlund transformation. They are supposed to satisfy
	the following system%
	\[
	\begin{cases}
		\left(  D_{x}-iD_{y}-\sqrt{3}D_{x}^{2}\right)  g\cdot h=0,\\
		\\
		\left(  D_{x}-iD_{y}+\sqrt{3}iD_{x}D_{y}-D_{x}^{3}\right)  g\cdot h=0.
	\end{cases}
	\]
	In the $z$-$\bar{z}$ coordinate, it takes the form:%
	\begin{equation}
		\label{gh}
		\begin{cases}
			\left(  2D_{z}-\sqrt{3}D_{x}^{2}\right)  g\cdot h=0,\\
			\\
			\left(  2D_{z}+\sqrt{3}iD_{x}D_{y}-D_{x}^{3}\right)  g\cdot h=0.
		\end{cases}
	\end{equation}
	
	The following result provides the explicit formula of a family of degree 6 polynomial solutions.
	\begin{proposition}
		\label{degree6}Let $\alpha,\beta$ be parameters and
		\begin{align*}
			h=~&z^{3}\bar{z}^{3}-2\sqrt{3}z^{2}\bar{z}^{3}+2\sqrt{3}z^{3}\bar{z}%
			^{2}-3z^{4}+15z^{2}\bar{z}^{2}+6z\bar{z}^{3}-3\bar{z}^{4}+6z^{3}\bar{z}\\
			&  +\beta z^{3}+24\sqrt{3}z^{2}\bar{z}-24\sqrt{3}z\bar{z}^{2}+\left(
			3\sqrt{3}+\alpha\right)  \bar{z}^{3}\\
			&  -\left(  90+2\sqrt{3}\right)  z^{2}+63z\bar{z}-\left(  72-2\sqrt{3}%
			\alpha\right)  \bar{z}^{2}\\
			&  +\left(  189\sqrt{3}-3\alpha+6\beta\right)  z+\left(  -180\sqrt{3}%
			+6\alpha-3\beta\right)  \bar{z}\\
			&  +1161-6\sqrt{3}\alpha+9\sqrt{3}\beta+\alpha\beta.
		\end{align*}
		Then $g,h$ satisfy $\left(  \ref{gh}\right)  .$
	\end{proposition}
	
	\begin{proof}
		The highest degree terms of $h$ has to be $h_{6}=z^{3}\bar{z}^{3}.$ The
		$h_{5}$ term can be obtained by solving the equation:
		\[
		2D_{z}g_{4}\cdot h_{5}=\sqrt{3}D_{x}^{2}g_{4}\cdot h_{6}-2D_{z}g_{3}\cdot
		h_{6},
		\]
		which gives
		\[
		h_{5}=-2\sqrt{3}z^{2}\bar{z}^{3}+cz^{3}\bar{z}^{2},
		\]
		where $c$ is a parameter to be determined using information of $h_{4},$ which
		satisfies
		\[
		2D_{z}g_{4}\cdot h_{4}+2D_{z}g_{3}\cdot h_{5}+2D_{z}g_{2}\cdot h_{6}-\sqrt
		{3}D_{x}^{2}g_{4}\cdot h_{5}-\sqrt{3}D_{x}^{2}g_{3}\cdot h_{6}=0,
		\]
		and%
		\begin{align*}
			&  2D_{z}g_{4}\cdot h_{4}+2D_{z}g_{3}\cdot h_{5}+2D_{z}g_{2}\cdot h_{6}%
			+\sqrt{3}iD_{x}D_{y}g_{4}\cdot h_{5}\\
			&  +\sqrt{3}iD_{x}D_{y}g_{3}\cdot h_{6}-D_{x}^{3}g_{4}\cdot h_{6}=0.
		\end{align*}
		The compatibility of these two equations gives $c=2\sqrt{3}\ $and hence
		\[
		h_{5}=-2\sqrt{3}z^{2}\bar{z}^{3}+2\sqrt{3}z^{3}\bar{z}^{2}.
		\]
		With $h_{5}$ at hand, $h_{4}$ can be found by solving the equation
		\[
		2D_{z}g_{4}\cdot h_{4}+2D_{z}g_{3}\cdot h_{5}+2D_{z}g_{2}\cdot h_{6}-\sqrt
		{3}D_{x}^{2}g_{4}\cdot h_{5}-\sqrt{3}D_{x}^{2}g_{3}\cdot h_{6}=0.
		\]
		This gives%
		\[
		h_{4}=-3z^{4}+15z^{2}\bar{z}^{2}+6z\bar{z}^{3}-3\bar{z}^{4}+cz^{3}\bar{z}.
		\]
		Again, $c$ is a parameter to be determined, using the equation of $h_{3}.$
		
		Now $h_{3}$ satisfies%
		\begin{align*}
			&  2D_{z}g_{4}\cdot h_{3}+2D_{z}g_{3}\cdot h_{4}+2D_{z}g_{2}\cdot h_{5}%
			+2D_{z}g_{1}\cdot h_{6}\\
			&  -\sqrt{3}D_{x}^{2}g_{4}\cdot h_{4}-\sqrt{3}D_{x}^{2}g_{3}\cdot h_{5}%
			-\sqrt{3}D_{x}^{2}g_{2}\cdot h_{6}=0,
		\end{align*}
		and%
		\begin{align*}
			&  2D_{z}g_{4}\cdot h_{3}+2D_{z}g_{3}\cdot h_{4}+2D_{z}g_{2}\cdot h_{5}%
			+2D_{z}g_{1}\cdot h_{6}+\sqrt{3}iD_{x}D_{y}g_{4}\cdot h_{4}\\
			&  +\sqrt{3}iD_{x}D_{y}g_{3}\cdot h_{5}+\sqrt{3}iD_{x}D_{y}g_{2}\cdot
			h_{6}-D_{x}^{3}g_{4}\cdot h_{5}-D_{x}^{3}g_{3}\cdot h_{6}=0.
		\end{align*}
		The compatibility implies that $c=6,$ and we deduce that
		\[
		h_{4}=-3z^{4}+15z^{2}\bar{z}^{2}+6z\bar{z}^{3}-3\bar{z}^{4}+6z^{3}\bar{z}.
		\]

		Having obtained $h_{4},$ we then proceed to solve the equation for $h_{3}$ and
		find that for some parameter $\beta,$%
		\[
		h_{3}=24\sqrt{3}z^{2}\bar{z}-24\sqrt{3}z\bar{z}^{2}+\left(  3+\frac{\alpha
		}{\sqrt{3}}\right)  \sqrt{3}\bar{z}^{3}+\beta z^{3}.
		\]
		The rest of terms $h_{2},h_{1},h_{0}$ follow from routine computation, and it
		turns out that $\beta$ is a free parameter. We omit the details.\ \ 
	\end{proof}
	
	Observe that the function $h$ given by Proposition \ref{degree6} is not real
	valued. But we are mainly interested in real valued solutions. In view of
	this, we replace $y$ by $y-\frac{2\sqrt{3}i}{3},$ and choose the complex
	parameters $\alpha,\beta$ such that
	\begin{align*}
		\alpha-\beta &  =Ai-\frac{211}{3\sqrt{3}},\\
		\alpha+\beta &  =B-3\sqrt{3},
	\end{align*}
	where $A,B$ are real numbers. Then the function $h$ becomes
	\begin{align*}
		h_{A,B}\left(  x,y\right)  :=~&x^{6}+3x^{4}y^{2}+3x^{2}y^{4}+y^{6}%
		+25x^{4}+90x^{2}y^{2}+17y^{4}\\
		&  +Bx^{3}+3Ax^{2}y-3Bxy^{2}-Ay^{3}-125x^{2}+475y^{2}\\
		&  -Bx+5Ay+1875+\frac{A^{2}}{4}+\frac{B^{2}}{4}.
	\end{align*}
	This is a family of real valued solution. Note that translation along the $x$
	or $y$ direction still yields a solution(the rotation will not), hence there
	are all together $4$ free real parameters in the whole family of solutions.
	Now if
	\[
	A=0\text{ and }B=0,
	\]
	then $h$ will be an even solution, which equals
	\begin{align*}
		&  x^{6}+3x^{4}y^{2}+3x^{2}y^{4}+y^{6}+25x^{4}+90x^{2}y^{2}+17y^{4}\\
		&  -125x^{2}+475y^{2}+1875.
	\end{align*}
	Note that this function is not radially symmetry. For general parameters
	$A,B$, the solution does not have any symmetry.
	
	We should emphasize that although $h_{A,B}$ is real valued, it is connected by
	Backlund transformation, via a degree $4$ polynomial, to the following degree
	$2$ polynomial:
	\[
	x^{2}+\left(  y-\frac{2\sqrt{3}i}{3}\right)  ^{2}+3,
	\]
	which is not real valued.
	
	\subsection{The general case}
	
	Our construction can then be iterated to created higher energy solutions(Note
	that energy is completely determined by the degree, \cite{Pelinovskii93}). One
	can indeed directly write down an algorithm to do this computation. At each
	stage, one find the solutions from their highest degree terms to lower degree
	terms. More precisely, suppose we have already found a polynomial solution $f$
	to the bilinear equation, with highest degree term equals $z^{n}\bar{z}^{n},$
	and we want to find its Backlund transformation $g.$ Then once we have found
	$g_{j}$ for $j>m,$ then from the Backlund transformation, we see that $g_{m}$
	will satisfy a system of equations of the form
	\begin{align*}
		D_{\bar{z}}\left(  z^{n}\bar{z}^{n}\right)  \cdot g_{m}  &  =\text{RHS}%
		_{1}\text{,}\\
		D_{\bar{z}}\left(  z^{n}\bar{z}^{n}\right)  \cdot g_{m}  &  =\text{RHS}%
		_{2}\text{,}%
	\end{align*}
	where RHS$_{1}$ and RHS$_{2}$ contain terms explicitly polynomial terms from
	$g_{j}$ with $j>m.$
	
	\medskip
	
	An important question is, in which form, should a free parameter appear.
	Assume that the term $g_{m+1}$ has a parameter term $\sigma z^{j}\bar{z}^{n}$
	to be determined, where $j+n=m+1.$ The compatibility of these two equations is
	RHS$_{1}-$RHS$_{2}=0.$ In this equation, the parameter $\sigma$ appears as
	\[
	\sigma\left(  D_{x}^{2}-iD_{x}D_{y}\right)  \left(  z^{n}\bar{z}^{n}\right)
	\cdot\left(  z^{j}\bar{z}^{n}\right)  .
	\]
	Direct computation tells us that this equals
	\[
	\sigma\left[  n\left(  n-1\right)  -2nj+j\left(  j-1\right)  \right]
	z^{j+n-2}\bar{z}^{2n}.
	\]
	Hence we conclude that free parameter(that is, no restriction on $\sigma$) can
	occur only if $j$ satisfies
	\[
	n\left(  n-1\right)  -2nj+j\left(  j-1\right)  =0.
	\]
	In the case of $n=3,$ the free parameter term is $\sigma z\bar{z}^{3}.$
	
	Next, let us assume that the polynomial solution $g,$ with highest degree term
	$z^{p}\bar{z}^{n}$ has been found. We would like to find function $h$ by
	another Backlund transformation. Similar as above, once we have found $h_{j}$
	for $j>m,$ then $h_{m}$ will satisfy a system of equations of the form
	\begin{align*}
		D_{z}\left(  z^{p}\bar{z}^{n}\right)  \cdot h_{m}  &  =\text{RHS}_{1}%
		\text{,}\\
		D_{z}\left(  z^{p}\bar{z}^{n}\right)  \cdot h_{m}  &  =\text{RHS}_{2}\text{.}%
	\end{align*}
	Assume that $h_{m+1}$ has a parameter term $\alpha z^{p}\bar{z}^{j}$ to be
	determined. Then $\alpha$ enters into the compatibility condition as
	\[
	\alpha\left(  D_{x}^{2}+iD_{x}D_{y}\right)  \left(  z^{p}\bar{z}^{n}\right)
	\cdot\left(  z^{p}\bar{z}^{j}\right)  .
	\]
	This equals%
	\[
	\alpha\left[  n\left(  n-1\right)  -2nj+j\left(  j-1\right)  \right]
	z^{2p}\bar{z}^{j+n-2}.
	\]
	Therefore, again, free parameter can occur only if
	\[
	n\left(  n-1\right)  -2nj+j\left(  j-1\right)  =0.
	\]
	In the case of $n=3,$ the free parameter term is $\alpha z^{6}\bar{z}.$
	
	\begin{remark}
		Our algorithm tells us that, at least formally, if we consider those
		solutions(complex valued) whose leading term is $z^{n}\bar{z}^{n},$ where
		$n=\frac{1}{2}k\left(  k+1\right)  ,$ then the space of these solutions should
		have complex dimension $2k.$ Moreover, the space of real valued solutions with
		degree $k\left(  k+1\right)  $ is expected to have real dimension $2k.$
	\end{remark}
	
	\section{Nondegeneracy and Morse index of degree $6$ solutions}
	
	In this section, we will show that the family of real valued solutions
	$u_{A,B}$ to the Boussinesq equation corresponding to $h_{A,B}$ have Morse
	index $3.$
	
	\medskip
	
	Our starting point in the computation of Morse index is to analyze the
	asymptotic behavior of $u=u_{0,B}$ for $B$ large.
	
	In view of the fact that
	\[
	u=2\partial_{x}^{2}\ln h_{0,B}=2\frac{h_{0,B}\partial_{x}^{2}h_{0,B}-\left(
		\partial_{x}h_{0,B}\right)  ^{2}}{h_{0,B}^{2}},
	\]
	for $B$ large, the maximum of $u$ should take place around the points $\left(
	x,y\right)  $ which solve the system of algebraic equations:
	\begin{equation}
		\label{peak}
		\begin{cases}
			\phi=0,\\
			\partial_{x}\phi=0,
		\end{cases}
	\end{equation}
	where $\phi$ designates the main order term of $h_{0,B}$(away from the maximum
	of $u$) and is defined by
	\[
	\phi\left(  x,y\right)  =x^{6}+3x^{4}y^{2}+3x^{2}y^{4}+y^{6}+Bx^{3}%
	-3Bxy^{2}+\frac{B^{2}}{4}.
	\]
	Solving $\left(  \ref{peak}\right)  $ and setting $\gamma=\left(  \frac{B}%
	{2}\right)  ^{\frac{1}{3}},$ we obtain the following three points $P_{j}$ on
	the $\left(  x,y\right)  $ plane:
	\begin{equation}
		P_{1}=\left(  -\gamma,0\right)  ,\text{ }P_{2}=\frac{1}{2}\left(
		\gamma,-\sqrt{3}\gamma\right)  ,\text{ }P_{3}=\frac{1}{2}\left(  \gamma
		,\sqrt{3}\gamma\right)  . \label{Pj}%
	\end{equation}
	These three points are the vertices of an equilateral triangle. This is in
	agreement with the formal computation for the dynamics of peaks of KP-I
	equation, carried out in \cite{Pelinovskii93}. The reason that $P_{j}$ are in
	this position will be clear later on.
	
	Let us set
	\[
	L\left(  x,y\right)  =x^{2}+y^{2}+3.
	\]
	Recall that we use $U$ to denote the classical lump solution. That is,%
	\begin{equation}
		U\left(  x,y\right)  =2\partial_{x}^{2}\ln L=4\frac{y^{2}-x^{2}+3}{\left(
			x^{2}+y^{2}+3\right)  ^{2}}. \label{lumpc}%
	\end{equation}
	Then we define
	\[
	L_{j}\left(  \cdot\right)  =L\left(  \cdot-P_{j}\right)  ,\text{ and }%
	U_{j}=2\partial_{x}^{2}L_{j},\text{ \ }j=1,2,3.
	\]
	
	\medskip
	The following result describes the asymptotic behavior of $u$ as $B$(or
	$\gamma$) tends to $+\infty.$
	
	\begin{lemma}
		\label{er}The error between $u$ and $U_{1}+U_{2}+U_{3}$ satisfies
		\[
		\left\Vert u-\left(  U_{1}+U_{2}+U_{3}\right)  \right\Vert _{L^{\infty}\left(
			\mathbb{R}^{2}\right)  }\rightarrow0,\text{ as }B\rightarrow+\infty.
		\]
		
	\end{lemma}
	
	\begin{proof}
		We have%
		\begin{align*}
			u-\left(  U_{1}+U_{2}+U_{3}\right)   &  =2\partial_{x}^{2}\left(  \ln h-\ln
			L_{1}-\ln L_{2}-\ln L_{3}\right) \\
			&  =-2\partial_{x}^{2}\left[  \ln\left(  1-\frac{h-L_{1}L_{2}L_{3}}{h}\right)
			\right]  .
		\end{align*}
		Direct computation tells us that
		\begin{align*}
			\eta  :=~& h-L_{1}L_{2}L_{3}\\
			=~&16x^{4}+72x^{2}y^{2}+8y^{4}-\left(  152+9\gamma^{2}\right)  x^{2}\\
			&  +\left(  448-9\gamma^{2}\right)  y^{2}-Bx+1848-27\gamma^{2}-9\gamma^{4}.
		\end{align*}
		Observe that at $P_{j},$ the main contribution to $h$ comes from those degree
		$4$ terms. Therefore, to estimate the error around $P_{j},$ we need to have a
		better control of $\eta$ at $P_{j}.$ It turns out that $\eta\left(
		P_{j}\right)  $ is of the order $O\left(  \gamma^{2}\right)  .$ More
		precisely,
		\begin{align*}
			\eta\left(  P_{1}\right) =-179\gamma^{2}+1848,\quad
			\eta\left(  P_{2}\right) =\eta\left(  P_{3}\right)  =271\gamma^{2}+1848.
		\end{align*}
		From these, we then conclude with little work that%
		\[
		\left\Vert \frac{h-L_{1}L_{2}L_{3}}{h}\right\Vert _{L^{\infty}\left(
			\mathbb{R}^{2}\right)  }\rightarrow0,
		\]
		which readily implies
		\[
		\left\Vert u-\left(  U_{1}+U_{2}+U_{3}\right)  \right\Vert _{L^{\infty}\left(
			\mathbb{R}^{2}\right)  }\rightarrow0,\text{ as }B\rightarrow+\infty.
		\]
		This finishes the proof.
	\end{proof}
	
	Lemma \ref{er} provides a rough picture of the solution $u.$ However, to
	analyze its Morse index, we need more precise expansion of $u.$ To obtain the
	required expansion, it turns out that the explicit form of the function $h$
	does not help too much. Therefore we use the mapping property of the
	linearized Boussinesq operator. This will be done in sequel.
	
	\medskip
	
	For a function $q,$ we use $\mathcal{L}_{q}$ to denote the linearized
	Boussinesq operator at $q,$ with the following form:
	\[
	\mathcal{L}_{q}\eta=\partial_{x}^{2}\eta-\eta+6q\eta-\partial_{x}^{-2}%
	\partial_{y}^{2}\eta.
	\]
	Here $\partial_{x}^{-1}=\int_{-\infty}^{x}.$ One of the reason that we
	integrate twice in the original form of the Boussinesq equation is that the
	operator $\mathcal{L}_{q}$ is self adjoint. We emphasize that in view of the
	definition of $\partial_{x}^{-1},$ one should be very careful about the
	integrability of the function.
	
	We let $\left(  x_{j}^{\ast},y_{j}^{\ast}\right)  $ be the point close to  $P_{j}=\left(  \tilde{x}_{j},\tilde{y}_{j}\right)  $, which is introduced in
	$\left(  \ref{Pj}\right)  .$ Setting
	\[
	U_{j}^{\ast}\left(  x,y\right)  :=U\left(  x-x_{j}^{\ast},y-y_{j}^{\ast
	}\right)
	\]
	and writing
	\[
	u=U^{\ast}+\xi,
	\]
	where $U^{\ast}$ is the \textquotedblleft approximate\textquotedblright%
	\ solution defined by
	\[
	U^{\ast}=U_{1}^{\ast}+U_{2}^{\ast}+U_{3}^{\ast},
	\]
	and $\xi$ is a perturbation term satisfying the following orthogonality
	condition: For $k=1,2,3,$%
	
	\[
	\int_{\mathbb{R}^{2}}\xi\partial_{x}U_{k}^{\ast}dxdy=0\text{ \ \ and \ }%
	\int_{\mathbb{R}^{2}}\xi\partial_{y}U_{k}^{\ast}dxdy=0.
	\]
	This can always be achieved by perturbing $\left(  \tilde{x}_{k},\tilde{y}%
	_{k}\right)  $ into $\left(  x_{k}^{\ast},y_{k}^{\ast}\right)  .$ Note that
	$\partial_{x}U,\partial_{y}U$ are kernels of the operator $\mathcal{L}_{U}.$
	Hence $\partial_{x}U_{k}^{\ast},\partial_{y}U_{k}^{\ast}$ are
	\textquotedblleft approximate\textquotedblright\ kernels of $\mathcal{L}%
	_{U^{\ast}}.$
	
	By Lemma \ref{er}, if the distance between $\left(  x_{j}^{\ast},y_{j}^{\ast
	}\right)  $ and $P_{j}$ is close enough, then $\left\Vert \xi\right\Vert
	_{L^{\infty}\left(  \mathbb{R}^{2}\right)  }$ will be small, provided that $B$
	is large. Since $u$ satisfies the Boussinesq equation, the perturbation $\xi$
	should satisfy the following nonlinear equation%
	\begin{equation}
		\mathcal{L}_{U^{\ast}}\xi=-E\left(  U^{\ast}\right)  -3\xi^{2}, \label{pert}%
	\end{equation}
	where $E\left(  U^{\ast}\right)  $ is the \textquotedblleft
	error\textquotedblright\ of the approximate solution $U^{\ast}:$
	\begin{align}
		E\left(  U^{\ast}\right)   &  =\partial_{x}^{2}U^{\ast}-U^{\ast}+3U^{\ast
			2}-\partial_{x}^{-2}\partial_{y}^{2}U^{\ast}\nonumber\\
		&  =6\left(  U_{1}^{\ast}U_{2}^{\ast}+U_{2}^{\ast}U_{3}^{\ast}+U_{1}^{\ast
		}U_{3}^{\ast}\right)  . \label{erofa}%
	\end{align}
	We see that essentially $E\left(  U^{\ast}\right)  $ gives the interaction
	between $U_{j}^{\ast},$ and therefore the presence of $\xi$ is due to this
	interaction. $E\left(  U^{\ast}\right)  $ is of the order $O\left(
	\gamma^{-2}\right)  .$
	
	We introduce the complex numbers $z_{j}^{\ast}=x_{j}^{\ast}+iy_{j}^{\ast
	},j=1,2,3,$ and define
	\begin{equation}
		d^{\ast}=24\int_{\mathbb{R}^{2}}U^{2}dxdy.\label{cstar}%
	\end{equation}
	An important ingredient of the analysis is to understand the projection of the
	error $E=E\left(  U^{\ast}\right)  $ onto the kernels. This is the content of
	the following
	
	\begin{lemma}
		\label{proj}There holds
		\begin{equation}
			\label{proony-1}
			\int_{\mathbb{R}^{2}}E\partial_{x}U_{j}^{\ast}dxdy=-d^{\ast}\sum\limits_{k\neq j}\operatorname{Re}
			\frac{1}{\left(  z_{j}^{\ast}-z_{k}^{\ast}\right)  ^{3}%
			}+O\left(  \gamma^{-4}\right) ,\quad j=1,2,3,
		\end{equation}
		and
		\begin{equation}
			\label{proony}
			\int_{\mathbb{R}^{2}}E\partial_{y}U_{j}^{\ast}dxdy=d^{\ast}\sum\limits_{k\neq j}\operatorname{Im}%
			\frac{1}{\left(  z_{j}^{\ast}-z_{k}^{\ast}\right)  ^{3}%
			}+O\left(  \gamma^{-4}\right),  \quad j=1,2,3.
		\end{equation}
		
	\end{lemma}
	
	\begin{proof}
		We shall prove \eqref{proony-1} and \eqref{proony} for the case $j=1$, the other ones can be treated similarly. For the former one, using \eqref{erofa} we get
		\[
		\int_{\mathbb{R}^{2}}E\partial_{x}U_{1}^{\ast}dxdy=6\int_{\mathbb{R}^{2}%
		}\left(  U_{1}^{\ast}U_{2}^{\ast}+U_{1}^{\ast}U_{3}^{\ast}+U_{2}^{\ast}%
		U_{3}^{\ast}\right)  \partial_{x}U_{1}^{\ast}dxdy.
		\]
		Let us compute each integral appeared in the right hand side. Integrating by
		parts, we get
		\begin{equation}
			\int_{\mathbb{R}^{2}}U_{2}^{\ast}U_{1}^{\ast}\partial_{x}U_{1}^{\ast
			}dxdy=-\frac{1}{2}\int_{\mathbb{R}^{2}}U_{1}^{\ast2}\partial_{x}U_{2}^{\ast
			}dxdy. \label{I1}%
		\end{equation}
		To estimate this integral, we need the identity
		\begin{equation}
			\frac{y^{2}-x^{2}}{\left(  x^{2}+y^{2}\right)  ^{2}}=-\frac{1}{2}\left(
			\frac{1}{z^{2}}+\frac{1}{\bar{z}^{2}}\right)  , \label{ident}
		\end{equation}
		which implies
		\begin{equation}\label{dy}
			\partial_{x}\left(  \frac{y^{2}-x^{2}}{\left(  x^{2}+y^{2}\right)  ^{2}%
			}\right)  =2\operatorname{Re}\frac{1}{z^{3}}
			,\quad \partial_{y}\left(  \frac{y^{2}-x^{2}}{\left(  x^{2}+y^{2}\right)  ^{2}%
			}\right)  =-2\operatorname{Im}\frac{1}{z^{3}},
		\end{equation}
		\[
		\partial_{x}^2\left(  \frac{y^{2}-x^{2}}{\left(  x^{2}+y^{2}\right)  ^{2}%
		}\right)  =-6\operatorname{Re}\frac{1}{z^{4}},
		\quad 
		\partial_{y}^2\left(  \frac{y^{2}-x^{2}}{\left(  x^{2}+y^{2}\right)  ^{2}%
		}\right)  =6\operatorname{Re}\frac{1}{z^{4}},
		\]
		and
		\[
		\label{dxy-1}
		\partial_{x}\partial_y\left(  \frac{y^{2}-x^{2}}{\left(  x^{2}+y^{2}\right)  ^{2}%
		}\right)  =6\operatorname{Im}\frac{1}{z^{4}}.	
		\]
		We then use the explicit form of the lump solution $U$ to deduce that within a
		ball of radius $\frac{\gamma}{2}$ centered at $P_{1},$ $\ $%
		\[
		\left\vert \partial_{x}U_{2}^{\ast}-8\operatorname{Re}\frac{1}{z^{3}%
		}\right\vert =O\left(  \gamma^{-4}\right)  .
		\]
		Inserting this estimate into $\left(  \ref{I1}\right)  ,$ we then get
		\[
		6\int_{\mathbb{R}^{2}}U_{2}^{\ast}U_{1}^{\ast}\partial_{x}U_{1}^{\ast
		}dxdy=-d^{\ast}\operatorname{Re}\frac{1}{\left(  z_{1}^{\ast}-z_{2}^{\ast
			}\right)  ^{3}}+O\left(  \gamma^{-4}\right)  ,
		\]
		where $d^{\ast}$ is defined in $\left(  \ref{cstar}\right)  .$
		
		Similarly,
		\[
		6\int_{\mathbb{R}^{2}}U_{3}^{\ast}U_{1}^{\ast}\partial_{x}U_{1}^{\ast
		}dxdy=-d^{\ast}\operatorname{Re}\frac{1}{\left(  z_{1}^{\ast}-z_{3}^{\ast
			}\right)  ^{3}}+O\left(  \gamma^{-4}\right)  .
		\]
		Moreover, we use the decay of $U_{2}^{\ast},U_{3}^{\ast}$ to conclude directly
		that
		\[
		\int_{\mathbb{R}^{2}}U_{2}^{\ast}U_{3}^{\ast}\partial_{x}U_{1}^{\ast
		}dxdy=O\left(  \gamma^{-4}\right)  .
		\]
		Combining all these estimates, we then get
		\[
		\int_{\mathbb{R}^{2}}E\partial_{x}U_{1}^{\ast}dxdy=-d^{\ast}\operatorname{Re}%
		\frac{1}{\left(  z_{1}^{\ast}-z_{2}^{\ast}\right)  ^{3}}-d^{\ast
		}\operatorname{Re}\frac{1}{\left(  z_{1}^{\ast}-z_{3}^{\ast}\right)  ^{3}%
		}+O\left(  \gamma^{-4}\right)  .
		\]

		The equation $\left(  \ref{proony}\right)  $ can be obtained in a very similar
		way, using the second identity of $\left(  \ref{dy}\right)  .$
	\end{proof}
	
	In view of Lemma \ref{proj}, the position $z_{j}^{\ast}$ of the single lumps
	should approximately satisfy the following balancing condition: For each fixed
	$j,$%
	\begin{equation}
		\sum\limits_{k\neq j}\frac{1}{\left(  z_{j}^{\ast}-z_{k}^{\ast}\right)  ^{3}%
		}=0. \label{balan}%
	\end{equation}
	As we have mentioned above, this condition has already been observed in
	Section 3 of \cite{Pelinovskii93} for the KP-I equation, from a more
	physically inspired point of view. On the other hand, the space $M$ of points
	$z_{j}^{\ast}$ satisfying these balancing equations $\left(  \ref{balan}%
	\right)  $ has been investigated in \cite{Moser}, where rational solutions of
	the KdV equation has been studied.
	
	Taking into account of  the previous computation, we shall define the map
	\[
	\mathcal{F}:\left(  z_{1},z_{2},z_{3}\right)  ^{T}\rightarrow\left(
	F_{1},F_{2},F_{3}\right)  ^{T},\text{ }z_{j}\in\mathbb{C},
	\]
	where
	\[
	F_{j}=%
	{\displaystyle\sum\limits_{k\neq j}}
	\frac{1}{\left(  z_{j}-z_{k}\right)  ^{3}},\text{ }j=1,2,3.
	\]
	Note that for $\tilde{z}_{1}:=-1,\tilde{z}_{2}:=\frac{1+\sqrt{3}i}{2}%
	,\tilde{z}_{3}:=\frac{1-\sqrt{3}i}{2},$ we have%
	\[
	\mathcal{F}\left(  \tilde{z}_{1},\tilde{z}_{2},\tilde{z}_{3}\right)  =0.
	\]
	The linearization of $\mathcal{F}$ will play important role in our later
	analysis. The derivative $D\mathcal{F}$ of $\mathcal{F}$ at $\left(
	z_{1},z_{2},z_{3}\right)  $ is a matrix of the form $\left[  F_{j,k}\right]
	,$ where $F_{j,k}=\partial_{z_{k}}F_{j}$. 
	
	The next lemma follows from direct
	computation of eigenvectors.
	
	\begin{lemma}
		\label{vector}The kernels of $M:=$ $D\mathcal{F}|_{\left(  \tilde{z}%
			_{1},\tilde{z}_{2},\tilde{z}_{3}\right)  }$ are given by
		\[
		c_{1}\mathbf{b}_{1}+c_{2}\mathbf{b}_{2},
		\]
		where $c_{1},c_{2}$ are complex numbers and
		\[
		\mathbf{b}_{1}=\left(  1,1,1\right)  ^{T},\text{ }\mathbf{b}_{2}=\left(
		\tilde{z}_{1},\tilde{z}_{2},\tilde{z}_{3}\right)  ^{T}.
		\]
		
	\end{lemma}
	
	We remark that the vector $\mathbf{b}_{1}$ reflects the translation invariance
	of the system, and $\mathbf{b}_{2}$ is corresponding to scaling and
	rotation(multiplication by a complex number $c$).
	
	\medskip
	To proceed, we use $d_{k}\left(
	x,y\right)  $ to denote the distance between $\left(  x,y\right)  $ and
	$\left(  x_{k}^{\ast},y_{k}^{\ast}\right)  .$ Let
	\[
	\theta_{\alpha}\left(  x,y\right)  =\left(  \sum\limits_{k=1}^{3}\left(
	1+d_{k}\right)  ^{-1}\right)  ^{-\alpha}.
	\]
	We need some apriori estimates for the linearized operator.
	
	\begin{lemma}
		\label{apries}Let $\varepsilon>0$ be a fixed small constant. Suppose $\eta$
		satisfies $\left\Vert \eta\theta_{\varepsilon}\right\Vert _{L^{\infty}\left(
			\mathbb{R}^{2}\right)  }<+\infty$ and
		\begin{equation}
			\mathcal{L}_{U^{\ast}}\eta=f, \label{etaf}%
		\end{equation}
		where $\left\Vert f\theta_{2}\right\Vert _{L^{\infty}\left(  \mathbb{R}%
			^{2}\right)  }<+\infty.$ Assume for $k=1,2,3,$
	\end{lemma}

	\[
	\int_{\mathbb{R}^{2}}\eta\partial_{x}U_{k}^{\ast}dxdy=0,\text{ \ and }%
	\int_{\mathbb{R}^{2}}\eta\partial_{y}U_{k}^{\ast}dxdy=0.
	\]
	Then for any $\sigma\in\left(  0,2-\varepsilon\right)  ,$ there holds
	\[
	\left\Vert \eta\theta_{\sigma}\right\Vert _{L^{\infty}\left(  \mathbb{R}%
		^{2}\right)  }\leq C\left\Vert f\theta_{2}\right\Vert _{L^{\infty}\left(
		\mathbb{R}^{2}\right)  },
	\]
	where $C$ is independent of $\sigma$ and $f.$
	
	\begin{proof}
		Consider the cone $\mathcal{C}_{1},\mathcal{C}_{2},\mathcal{C}_{3}$ with
		vertex at the origin containing those points whose angle coordinate are in the
		range $\left(  \frac{2\pi}{3},\frac{4\pi}{3}\right)  ,\left(  \frac{2\pi}%
		{3},2\pi\right)  ,\left(  0,\frac{2\pi}{3}\right)  ,$ respectively. Then
		$P_{k}\in\mathcal{C}_{k}.$ Let $\rho_{1},\rho_{2},\rho_{3}$ be a partition of
		unity such that $\rho_{k}$ equals $1$ in most part of $\mathcal{C}_{k}$ and
		$\nabla\rho_{k}$ is supported in a radius $1$ tubular neighborhood of
		$\partial\mathcal{C}_{k}.$
		
		We rewrite the equation $\left(  \ref{etaf}\right)  $ into the form
		\[
		\mathcal{L}_{0}\eta:=\partial_{x}^{2}\eta-\eta-\partial_{x}^{-2}\partial
		_{y}^{2}\eta=f-6U^{\ast}\eta.
		\]
		Hence $\eta=\eta_{1}+\eta_{2}+\eta_{3},$ where $f_{k}=\rho_{k}f$ and $\eta
		_{k}$ is determined by the equation
		\[
		\mathcal{L}_{0}\eta_{k}=f_{k}-6U_{k}^{\ast}\eta.
		\]
		Observe that $\mathcal{L}_{0}$ is a hypo-elliptic operator with constant
		coefficient, and the decay properties of its Green function $K$ have been
		established in \cite{Bouard97}. In particular,
		\[
		\left\Vert r^{2}K\right\Vert _{L^{\infty}\left(  \mathbb{R}^{2}\right)  }\leq
		C.
		\]
		Using this decay estimate, we obtain
		\[
		\left\Vert \left(  1+d_{k}\right)  ^{\sigma}\eta_{k}\right\Vert _{L^{\infty
			}\left(  \mathbb{R}^{2}\right)  }\leq C\left\Vert \left(  1+d_{k}\right)
		^{2}f_{k}\right\Vert _{L^{\infty}\left(  \mathbb{R}^{2}\right)  }+C\left\Vert
		\left(  1+d_{k}\right)  ^{\sigma}U_{k}^{\ast}\eta\right\Vert _{L^{\infty
			}\left(  \mathbb{R}^{2}\right)  }.
		\]
		It follows that for some $r_{0}$ sufficiently large,
		\[
		\left\Vert \left(  1+d_{k}\right)  ^{\sigma}\eta_{k}\right\Vert _{L^{\infty
			}\left(  \mathbb{R}^{2}\right)  }\leq C\left\Vert \left(  1+d_{k}\right)
		^{2}f_{k}\right\Vert _{L^{\infty}\left(  \mathbb{R}^{2}\right)  }+C\left\Vert
		\eta\right\Vert _{L^{\infty}\left(  B_{r_{0}}\left(  P_{k}\right)  \right)
		}.
		\]
		Now we claim%
		\[
		\left\Vert \eta\right\Vert _{L^{\infty}\left(  B_{r_{0}}\left(  P_{k}\right)
			\right)  }\leq C\left\Vert \left(  1+d_{k}\right)  ^{2}f_{k}\right\Vert
		_{L^{\infty}\left(  \mathbb{R}^{2}\right)  }.
		\]
		Otherwise, there would exist a sequence of $\eta\left(  x-x_{k}^{\ast}%
		,y-y_{k}^{\ast}\right)  $ which converges to solution of the equation
		\[
		\mathcal{L}_{U}\eta_{\infty}=0.
		\]
		However, this contradicts with the nondegeneracy of lump solution \cite{Liu}
		and the assumption that $\eta$ is orthogonal to the $\partial_{x}U_{k}^{\ast
		},\partial_{y}U_{k}^{\ast}.$
	\end{proof}

	The next lemma deals with the explicit expression of the function related to the main order correction of the approximate solution.
	\begin{lemma}
		\label{aprio}Let $U$ be the classical lump solution defined by $\left(
		\ref{lumpc}\right)  $ and
		\[
		\omega=\partial_{x}\left[  \frac{24x\left(  y^{2}-3\right)  }{\left(
			x^{2}+y^{2}+3\right)  ^{2}}\right]  .\
		\]
		Then
		\begin{equation}
			\mathcal{L}_{U}\omega=-6U,\text{ and }\mathcal{L}_{U}\left[  \partial
			_{x}\omega\right]  =-6\partial_{x}U-6\partial_{x}U\omega. \label{Lw}%
		\end{equation}
		
	\end{lemma}
	
	\begin{proof}
		This follows from direct computation. Note that $\omega$ decays at the rate
		$O\left(  r^{-2}\right)  $ at infinity.
	\end{proof}
	
	With Lemma \ref{aprio} being understood, we introduce the notation
	\[
	\mathbf{p}_{k}=-2\sum\limits_{j\neq k}\operatorname{Re}\frac{1}{\left(
		z_{k}^{\ast}-z_{j}^{\ast}\right)  ^{2}},\text{ \ \ for \ \ }k=1,2,3.
	\]
	From the explicit formula of $U$ and $\left(  \ref{ident}\right)  ,$ we
	deduce
	\[
	U_{2}^{\ast}\left(  x_{1}^{\ast},y_{1}^{\ast}\right)  +U_{3}^{\ast}\left(
	x_{1}^{\ast},y_{1}^{\ast}\right)  =\mathbf{p}_{1}+O\left(  \gamma^{-4}\right)
	,
	\]
	and $\left\vert \mathbf{p}_{k}\right\vert =O\left(  \gamma^{-2}\right)  .$  We then define, for\ $k=1,2,3,$
	\[
	\omega_{k}\left(  x,y\right)  :=\mathbf{p}_{k}\omega\left(  x-x_{k}^{\ast
	},y-y_{k}^{\ast}\right)  .
	\]

	With the help of  $\omega_{k}$, we will prove the following
	result, which gives us a more precise description of the solution $u.$
	
	\begin{proposition}
		There exists $\varepsilon>0$ such that the function $\xi=u-U^{\ast}$ has the
		following expansion:%
		\[
		\left\Vert \left(  \xi-\sum\limits_{k=1}^{3}\omega_{k}\right)  \theta
		_{\varepsilon}\right\Vert _{L^{\infty}\left(  \mathbb{R}^{2}\right)  }\leq
		C\gamma^{-3}.
		\]
		
	\end{proposition}
	
	\begin{proof}
		We write $\xi=\sum\limits_{k=1}^{3}\omega_{k}+\eta,$ then $\eta$ satisfies
		\[
		\mathcal{L}_{U^{\ast}}\eta=-\sum\limits_{k=1}^{3}\mathcal{L}_{U^{\ast}}%
		\omega_{k}-E\left(  U^{\ast}\right)  -3\left(  \sum\limits_{k=1}^{3}\omega
		_{k}+\eta\right)  ^{2}.
		\]
		We have
		\[
		\left\Vert \left(  -\sum\limits_{k=1}^{3}\mathcal{L}_{U^{\ast}}\omega
		_{k}-E\left(  U^{\ast}\right)  \right)  \theta_{2}\right\Vert _{L^{\infty
			}\left(  \mathbb{R}^{2}\right)  }\leq C\gamma^{-3}.
		\]
		Applying Lemma \ref{apries}, we conclude that for some $\varepsilon>0,$
		\[
		\left\Vert \eta\theta_{\varepsilon}\right\Vert _{L^{\infty}\left(
			\mathbb{R}^{2}\right)  }\leq C\gamma^{-3}.
		\]
		This finishes the proof.
	\end{proof}
	
	In the next result, we need to use the following constants:
	\begin{align*}
		a^{\ast}  &  =\int_{\mathbb{R}^{2}}\left(  3\omega^{2}+6\omega\right)
		\partial_{x}^{2}Udxdy,\\
		b^{\ast}  &  =\int_{\mathbb{R}^{2}}\left(  3\omega^{2}+6\omega\right)
		\partial_{x}\partial_{y}Udxdy,\\
		c^{\ast} & = \int_{\mathbb{R}^{2}}\left(  3\omega^{2}+6\omega\right)
		\partial_{y}^2Udxdy.
	\end{align*}

	\begin{lemma}
		\label{pro1}
		For any index $j\in\{1,2,3\}$ it holds that
		\begin{equation}
			\label{lemma13-1}
			\begin{aligned}
				&\int_{\mathbb{R}^{2}}\partial_{x}U_{j}^{\ast}\mathcal{L}_{u}\left[
				\partial_{x}U_{j}^{\ast}\right]  dxdy=-3d^{\ast}\sum_{k\neq j}\operatorname{Re}\left[
				\left( z_{j}^{\ast}%
				-z_{k}^{\ast}\right)  ^{-4}\right]  +a^{\ast}\mathbf{p}_{j}^{2}+O\left(
				\gamma^{-5}\right)  ,	\\
				&\int_{\mathbb{R}^{2}}\partial_{x}U_{j}^{\ast}\mathcal{L}_{u}\left[
				\partial_{y}U_{j}^{\ast}\right]  dxdy=3d^{\ast}\sum_{k\neq j}\operatorname{Im}\left[
				\left(  z_{j}^{\ast}%
				-z_{k}^{\ast}\right)  ^{-4}\right]  +b^{\ast}\mathbf{p}_{j}^{2}+O\left(
				\gamma^{-5}\right)  ,\\
				&\int_{\mathbb{R}^{2}}\partial_{y}U_{j}^{\ast}\mathcal{L}_{u}\left[
				\partial_{y}U_{j}^{\ast}\right]  dxdy=3d^{\ast}\sum_{k\neq j}\operatorname{Re}\left[
				\left(  z_{j}^{\ast}%
				-z_{k}^{\ast}\right)  ^{-4}\right]  +c^{\ast}\mathbf{p}_{j}^{2}+O\left(
				\gamma^{-5}\right).
			\end{aligned}
		\end{equation}
		While for different indices $j,k\in\{1,2,3\}$, we have 
		\begin{equation}
			\label{lemma13-2}
			\begin{aligned}
				\int_{\mathbb{R}^{2}}\partial_{x}U_{j}^{\ast}\mathcal{L}_{u}\left[
				\partial_{x}U_{k}^{\ast}\right]  dxdy=3d^{\ast}\operatorname{Re}\left[
				\left(  z_{j}^{\ast}-z_{k}^{\ast}\right)  ^{-4}\right]  +O\left(  \gamma
				^{-5}\right),\\
				\int_{\mathbb{R}^{2}}\partial_{x}U_{j}^{\ast}\mathcal{L}_{u}\left[
				\partial_{y}U_{k}^{\ast}\right]  dxdy=-3d^{\ast}\operatorname{Im}\left[
				\left(  z_{j}^{\ast}-z_{k}^{\ast}\right)  ^{-4}\right]  +O\left(  \gamma
				^{-5}\right),\\
				\int_{\mathbb{R}^{2}}\partial_{y}U_{j}^{\ast}\mathcal{L}_{u}\left[
				\partial_{y}U_{k}^{\ast}\right]  dxdy=-3d^{\ast}\operatorname{Re}\left[
				\left(  z_{j}^{\ast}-z_{k}^{\ast}\right)  ^{-4}\right]  +O\left(  \gamma
				^{-5}\right).
			\end{aligned}	
		\end{equation}
	\end{lemma}
	
	\begin{proof}
		We shall firstly verify the three equations in \eqref{lemma13-1}. Without loss of generality we may assume that $j=1.$ Since $\partial_{x}U_{1}^{\ast}$ is a kernel of the operator $\mathcal{L}%
		_{U_{1}^{\ast}},$ we have%
		\begin{equation}
			\mathcal{L}_{u}\left[  \partial_{x}U_{1}^{\ast}\right]  =\mathcal{L}%
			_{U_{1}^{\ast}}\left[  \partial_{x}U_{1}^{\ast}\right]  +6\left(
			u-U_{1}^{\ast}\right)  \partial_{x}U_{1}^{\ast}=6\left(  U_{2}^{\ast}%
			+U_{3}^{\ast}+\xi\right)  \partial_{x}U_{1}^{\ast}.\label{LU1}%
		\end{equation}
		As a direct consequence,
		\begin{equation}
			\int_{\mathbb{R}^{2}}\partial_{x}U_{1}^{\ast}\mathcal{L}_{u}\left[
			\partial_{x}U_{1}^{\ast}\right]  dxdy=6\int_{\mathbb{R}^{2}}\left(
			\partial_{x}U_{1}^{\ast}\right)  ^{2}\left(  U_{2}^{\ast}+U_{3}^{\ast}%
			+\xi\right)  dxdy.\label{u1u1}%
		\end{equation}
		To estimate the right hand side, we first differentiate the equation
		\[
		\mathcal{L}_{U^{\ast}}\xi=-E\left(  U^{\ast}\right)  -3\xi^{2}:=J
		\]
		with respect to $x$. This yields
		\[
		\mathcal{L}_{U^{\ast}}\left[  \partial_{x}\xi\right]  +6\partial_{x}U^{\ast
		}\xi=\partial_{x}J.
		\]
		Multiplying both sides with $\partial_{x}U_{1}^{\ast}$ and integrating by
		parts, we obtain
		\[
		\int_{\mathbb{R}^{2}}\left(  -\partial_{x}\left(  \mathcal{L}_{U^{\ast}%
		}\left[  \partial_{x}U_{1}^{\ast}\right]  \right)  +6\partial_{x}U_{1}^{\ast
		}\partial_{x}U^{\ast}\right)  \xi dxdy=\int_{\mathbb{R}^{2}}\partial
		_{x}J\partial_{x}U_{1}^{\ast}dxdy.
		\]
		Reorganizing terms, we have
		\begin{align*}
			6\int_{\mathbb{R}^{2}}\left(  \partial_{x}U_{1}^{\ast}\right)  ^{2}\xi dxdy 
			=&\int_{\mathbb{R}^{2}}\partial_{x}J\partial_{x}U_{1}^{\ast}dxdy-6\int%
			_{\mathbb{R}^{2}}\left(  \partial_{x}U_{2}^{\ast}+\partial_{x}U_{3}^{\ast
			}\right)  \partial_{x}U_{1}^{\ast}\xi dxdy\\
			&  +6\int_{\mathbb{R}^{2}}\partial_{x}\left[  \left(  U_{2}^{\ast}+U_{3}%
			^{\ast}\right)  \partial_{x}U_{1}^{\ast}\right]  \xi dxdy\\
			=&\int_{\mathbb{R}^{2}}\partial_{x}J\partial_{x}U_{1}^{\ast}dxdy+6\int%
			_{\mathbb{R}^{2}}\left(  U_{2}^{\ast}+U_{3}^{\ast}\right)  \partial_{x}%
			^{2}U_{1}^{\ast}\xi dxdy.
		\end{align*}
		Inserting this into $\left(  \ref{u1u1}\right)  ,$ we get
		\begin{align*}
			&  \int_{\mathbb{R}^{2}}\partial_{x}U_{1}^{\ast}\mathcal{L}_{u}\left[
			\partial_{x}U_{1}^{\ast}\right]  dxdy\\
			&  =6\int_{\mathbb{R}^{2}}\left(  \partial_{x}U_{1}^{\ast}\right)  ^{2}\left(
			U_{2}^{\ast}+U_{3}^{\ast}\right)  dxdy+\int_{\mathbb{R}^{2}}\partial
			_{x}J\partial_{x}U_{1}^{\ast}dxdy+6\int_{\mathbb{R}^{2}}\left(  U_{2}^{\ast
			}+U_{3}^{\ast}\right)  \partial_{x}^{2}U_{1}^{\ast}\xi dxdy\\
			&  =3\int_{\mathbb{R}^{2}}U_{1}^{\ast2}\partial_{x}^{2}\left(  U_{2}^{\ast
			}+U_{3}^{\ast}\right)  dxdy-3\int_{\mathbb{R}^{2}}\partial_{x}\left(  \xi
			^{2}\right)  \partial_{x}U_{1}^{\ast}dxdy\\
			&\quad +6\int_{\mathbb{R}^{2}}\left(  U_{2}^{\ast}+U_{3}^{\ast}\right)
			\partial_{x}^{2}U_{1}^{\ast}\xi dxdy+O\left(  \gamma^{-5}\right)  .
		\end{align*}
		Note that
		\[
		\int_{\mathbb{R}^{2}}U_{1}^{\ast2}\partial_{x}^{2}\left(  U_{2}^{\ast}%
		+U_{3}^{\ast}\right)  dxdy=-d^{\ast}\operatorname{Re}\left[  \left(
		z_{1}^{\ast}-z_{2}^{\ast}\right)  ^{-4}+\left(  z_{1}^{\ast}-z_{3}^{\ast
		}\right)  ^{-4}\right]  +O\left(  \gamma^{-5}\right)  .
		\]
		We then get
		\begin{align*}
			\int_{\mathbb{R}^{2}}\partial_{x}U_{1}^{\ast}\mathcal{L}_{u}\left[
			\partial_{x}U_{1}^{\ast}\right]  dxdy
			=&-3d^{\ast}\operatorname{Re}\left[  \left(  z_{1}^{\ast}-z_{2}^{\ast
			}\right)  ^{-4}+\left(  z_{1}^{\ast}-z_{3}^{\ast}\right)  ^{-4}\right]  \\
			&  +\mathbf{p}_{1}^{2}\int_{\mathbb{R}^{2}}\left(  3\omega_{1}^{2}+6\omega
			_{1}\right)  \partial_{x}^{2}U_{1}^{\ast}dxdy+O\left(  \gamma^{-5}\right)  .
		\end{align*}
		This is the required identity.
		
		\medskip
		
		Next we compute
		\begin{align*}
			&  \int_{\mathbb{R}^{2}}\partial_{x}U_{1}^{\ast}\mathcal{L}_{u}\left[
			\partial_{y}U_{1}^{\ast}\right]  dxdy\\
			&  =\int_{\mathbb{R}^{2}}\partial_{x}U_{1}^{\ast}\mathcal{L}_{U_{1}^{\ast}%
			}\left[  \partial_{y}U_{1}^{\ast}\right]  dxdy+6\int_{\mathbb{R}^{2}}%
			\partial_{x}U_{1}^{\ast}\left(  u-U_{1}^{\ast}\right)  \partial_{y}U_{1}%
			^{\ast}dxdy\\
			&  =6\int_{\mathbb{R}^{2}}\partial_{x}U_{1}^{\ast}\partial_{y}U_{1}^{\ast
			}\left(  U_{2}^{\ast}+U_{3}^{\ast}+\xi\right)  dxdy.
		\end{align*}
		On the other hand,
		\[
		\mathcal{L}_{U^{\ast}}\left[  \partial_{y}\xi\right]  +6\partial_{y}U^{\ast
		}\xi=\partial_{y}J,
		\]
		which implies
		\[
		\int_{\mathbb{R}^{2}}\left(  -\partial_{y}\left(  \mathcal{L}_{U^{\ast}%
		}\left[  \partial_{x}U_{1}^{\ast}\right]  \right)  +6\partial_{x}U_{1}^{\ast
		}\partial_{y}U^{\ast}\right)  \xi dxdy=\int_{\mathbb{R}^{2}}\partial
		_{y}J\partial_{x}U_{1}^{\ast}dxdy.
		\]
		From this identity, we get, using similar computation as before,
		\begin{align*}
			&  \int_{\mathbb{R}^{2}}\partial_{x}U_{1}^{\ast}\mathcal{L}_{u}\left[
			\partial_{y}U_{1}^{\ast}\right]  dxdy\\
			&  =6\int_{\mathbb{R}^{2}}\partial_{x}U_{1}^{\ast}\partial_{y}U_{1}^{\ast
			}\left(  U_{2}^{\ast}+U_{3}^{\ast}\right)  dxdy+\int_{\mathbb{R}^{2}}%
			\partial_{y}J\partial_{x}U_{1}^{\ast}dxdy+6\int_{\mathbb{R}^{2}}\left(
			U_{2}^{\ast}+U_{3}^{\ast}\right)  \partial_{x}\partial_{y}U_{1}^{\ast}\xi
			dxdy\\
			&  =3\int_{\mathbb{R}^{2}}\partial_{x}\partial_{y}\left(  U_{2}^{\ast}%
			+U_{3}^{\ast}\right)  U_{1}^{\ast2}dxdy-3\int_{\mathbb{R}^{2}}\partial
			_{y}\left(  \xi^{2}\right)  \partial_{x}U_{1}^{\ast}dxdy\\
			&\quad+6\int_{\mathbb{R}%
				^{2}}\left(  U_{2}^{\ast}+U_{3}^{\ast}\right)  \partial_{x}\partial_{y}%
			U_{1}^{\ast}\xi dxdy+O\left(  \gamma^{-5}\right)  .
		\end{align*}
		In view of the estimate
		\[
		\int_{\mathbb{R}^{2}}U_{1}^{\ast2}\partial_{x}\partial_{y}\left(  U_{2}^{\ast
		}+U_{3}^{\ast}\right)  dxdy=-3d^{\ast}\operatorname{Im}\left[  \left(
		z_{1}^{\ast}-z_{2}^{\ast}\right)  ^{-4}+\left(  z_{1}^{\ast}-z_{3}^{\ast
		}\right)  ^{-4}\right]  +O\left(  \gamma^{-5}\right)  ,
		\]
		we then arrive at
		\begin{align*}
			\int_{\mathbb{R}^{2}}\partial_{x}U_{1}^{\ast}\mathcal{L}_{u}\left[
			\partial_{y}U_{1}^{\ast}\right]  dxdy=~&3d^{\ast}\operatorname{Im}\left[  \left(  z_{1}^{\ast}-z_{2}^{\ast
			}\right)  ^{-4}+\left(  z_{1}^{\ast}-z_{3}^{\ast}\right)  ^{-4}\right]  \\
			&  +\mathbf{p}_{1}^{2}\int_{\mathbb{R}^{2}}\left(  3\omega_{1}^{2}+6\omega
			_{1}\right)  \partial_{x}\partial_{y}U_{1}^{\ast}dxdy+O\left(  \gamma
			^{-5}\right)  .
		\end{align*}
		
		\medskip
		
		Regarding the last equation in \eqref{lemma13-1}, as what we have done for the first one we get that
		\begin{align*}
			&  \int_{\mathbb{R}^{2}}\partial_{y}U_{1}^{\ast}\mathcal{L}_{u}\left[
			\partial_{y}U_{1}^{\ast}\right]  dxdy\\
			&  =6\int_{\mathbb{R}^{2}}\left(  \partial_{y}U_{1}^{\ast}\right)  ^{2}\left(
			U_{2}^{\ast}+U_{3}^{\ast}\right)  dxdy+\int_{\mathbb{R}^{2}}\partial
			_{y}J\partial_{y}U_{1}^{\ast}dxdy+6\int_{\mathbb{R}^{2}}\left(  U_{2}^{\ast
			}+U_{3}^{\ast}\right)  \partial_{y}^{2}U_{1}^{\ast}\xi dxdy\\
			&  =3\int_{\mathbb{R}^{2}}U_{1}^{\ast2}\partial_{y}^{2}\left(  U_{2}^{\ast
			}+U_{3}^{\ast}\right)  dxdy-3\int_{\mathbb{R}^{2}}\partial_{y}\left(  \xi
			^{2}\right)  \partial_{y}U_{1}^{\ast}dxdy\\
			&\quad+6\int_{\mathbb{R}^{2}}\left(  U_{2}^{\ast}+U_{3}^{\ast}\right)
			\partial_{y}^{2}U_{1}^{\ast}\xi dxdy+O\left(  \gamma^{-5}\right).
		\end{align*}
		Using the fact that 
		\[
		\int_{\mathbb{R}^{2}}U_{1}^{\ast2}\partial_{y}^{2}\left(  U_{2}^{\ast}%
		+U_{3}^{\ast}\right)  dxdy= d^{\ast}\operatorname{Re}\left[  \left(
		z_{1}^{\ast}-z_{2}^{\ast}\right)  ^{-4}+\left(  z_{1}^{\ast}-z_{3}^{\ast
		}\right)  ^{-4}\right]  +O\left(  \gamma^{-5}\right)  .
		\]
		We obtain that
		\begin{align*}
			\int_{\mathbb{R}^{2}}\partial_{y}U_{1}^{\ast}\mathcal{L}_{u}\left[
			\partial_{y}U_{1}^{\ast}\right]  dxdy=~&3d^{\ast}\operatorname{Re}\left[  \left(  z_{1}^{\ast}-z_{2}^{\ast
			}\right)  ^{-4}+\left(  z_{1}^{\ast}-z_{3}^{\ast}\right)  ^{-4}\right]  \\
			&  +\mathbf{p}_{1}^{2}\int_{\mathbb{R}^{2}}\left(  3\omega_{1}^{2}+6\omega
			_{1}\right)  \partial_{y}^{2}U_{1}^{\ast}dxdy+O\left(  \gamma^{-5}\right)  .
		\end{align*}
		
		Now we verify the equations in \eqref{lemma13-2}. 
		\begin{align*}
			&  \int_{\mathbb{R}^{2}}\partial_{x}U_{j}^{\ast}\mathcal{L}_{u}\left[
			\partial_{x}U_{k}^{\ast}\right]  dxdy\\
			&  =\int_{\mathbb{R}^{2}}\partial_{x}U_{j}^{\ast}\mathcal{L}_{U_{k}^{\ast}%
			}\left[  \partial_{x}U_{k}^{\ast}\right]  dxdy+6\int_{\mathbb{R}^{2}}%
			\partial_{x}U_{j}^{\ast}\left(  u-U_{k}^{\ast}\right)  \partial_{x}U_{k}%
			^{\ast}dxdy\\
			&  =6\int_{\mathbb{R}^{2}}\partial_{x}U_{j}^{\ast}\partial_{x}U_{k}^{\ast
			}\left(\sum_{\ell\neq k}U_{\ell}^{\ast}+\xi\right)  dxdy\\
			&  =6\int_{\mathbb{R}^{2}}\partial_{x}U_{j}^{\ast}\partial_{x}U_{k}^{\ast
			}U_{j}^{\ast}dxdy+O\left(  \gamma^{-5}\right)  \\
			&  =3d^{\ast}\operatorname{Re}\left[  \left(  z_{j}^{\ast}-z_{k}^{\ast
			}\right)  ^{-4}\right]  +O\left(  \gamma^{-5}\right)  .
		\end{align*}
		Since the computation of other integrals like $\int_{\mathbb{R}^{2}}%
		\partial_{x}U_{j}^{\ast}\mathcal{L}_{u}\left[  \partial_{y}U_{k}^{\ast
		}\right]  dxdy$ is quite similar, the details of these computation will be omitted.
	\end{proof}
	
	The key result of this section is the following
	
	\begin{proposition}
		\label{Morse}The Morse index of $u$ is equal to $3,$ provided that $B$ is
		sufficiently large.
	\end{proposition}
	
	\begin{proof}
		Let $\lambda_{B}$ be a negative eigenvalue of the linearized operator, with
		$\phi_{B}$ being an eigenfunction normalized such that $\left\Vert \phi
		_{B}\right\Vert _{L^{\infty}\left(  \mathbb{R}^{2}\right)  }=1.$ Then
		\[
		\partial_{x}^{2}\phi_{B}-\phi_{B}+6u\phi_{B}-\partial_{x}^{-2}\partial_{y}%
		^{2}\phi_{B}=\lambda_{B}\phi_{B}.
		\]
		Taking $x$-derivative twice, we get
		\[
		\partial_{x}^{2}\left(  \partial_{x}^{2}\phi_{B}-\phi_{B}+6u\phi_{B}\right)
		-\partial_{y}^{2}\phi_{B}=\lambda_{B}\partial_{x}^{2}\phi_{B}.
		\]

		Our first step is to show that there exists $c_{0}$ independent of $B$ such
		that
		\begin{equation}
			\lambda_{B}\leq c_{0}<0. \label{bound}%
		\end{equation}
		Assume to the contrary that $\left(  \ref{bound}\right)  $ was not true. Then
		there was a sequence $\lambda_{j}\rightarrow0$ with corresponding normalized
		eigenfunctions $\phi_{j}.$
		
		Consider the translated sequence $\tilde{\phi}_{j}\left(  x,y\right)
		:=\phi_{j}\left(  x-x_{1}^{\ast},y-y_{1}^{\ast}\right)  .$ Since we have
		assumed that $\left\Vert \phi_{j}\right\Vert _{L^{\infty}\left(
			\mathbb{R}^{2}\right)  }=1,$ the new functions $\tilde{\phi}_{j}$ will
		converge to a function $\Phi,$ solution of the equation
		\[
		\partial_{x}^{2}\left(  \partial_{x}^{2}\Phi-\Phi+6U\Phi\right)  -\partial
		_{y}^{2}\Phi=0.
		\]
		By the nondegeneracy of lump, there exist constants $c_{1},c_{2}$(could be
		zero) such that
		\[
		\Phi=c_{1}\partial_{x}U+c_{2}\partial_{y}U.
		\]

		To simplify the notation, we omit the subscript $j$, if no confusion will
		arise. Then we can write $\phi$ as $\phi^{\ast}+\zeta,$ where
		\[
		\phi^{\ast}=%
		{\displaystyle\sum\limits_{k=1}^{3}}
		\left(  \sigma_{k}\partial_{x}U_{k}^{\ast}+\tau_{k}\partial_{y}U_{k}^{\ast
		}\right)  ,
		\]
		and $\zeta$ is small and satisfies the orthogonality condition: For $k=1,2,3,$%
		\[
		\int_{\mathbb{R}^{2}}\zeta\partial_{x}U_{k}^{\ast}dxdy=0,\text{ \ and }%
		\int_{\mathbb{R}^{2}}\zeta\partial_{y}U_{k}^{\ast}dxdy=0.
		\]
		We would like to estimate $\zeta,$ using the equation
		\[
		\mathcal{L}_{u}\zeta=-\mathcal{L}_{u}\phi^{\ast}+\lambda\phi^{\ast}%
		+\lambda\zeta.
		\]
		To do this, we first observe that $\mathcal{L}_{u}\phi^{\ast}$ is of the order
		$O\left(  \gamma^{-2}\right)  .$ Hence to obtain a precise expansion for
		$\zeta,$ we write
		\[
		\zeta=\sum\limits_{k=1}^{3}\left(  \sigma_{k}\partial_{x}\omega_{k}+\tau
		_{k}\partial_{y}\omega_{k}\right)  +\zeta^{\ast}.
		\]
		Using $\left(  \ref{Lw}\right)  ,\left(  \ref{LU1}\right)  ,$ for any $k=1,2,3,$ we can estimate%
		\begin{equation}
			\mathcal{L}_{u}\left[  \partial_{x}U_{k}^{\ast}\right]  +\mathcal{L}%
			_{u}\left[  \partial_{x}\omega_{k}\right]  =O\left(  \gamma^{-3}\right), 
			\label{e1}%
		\end{equation}%
		\begin{equation}
			\mathcal{L}_{u}\left[  \partial_{y}U_{k}^{\ast}\right]  +\mathcal{L}%
			_{u}\left[  \partial_{y}\omega_{k}\right]  =O\left(  \gamma^{-3}\right). 
			\label{e11}%
		\end{equation}
		It follows that
		\begin{align*}
			\mathcal{L}_{u}\zeta^{\ast}-\lambda\xi^{\ast}    =&-\sum\limits_{k=1}%
			^{3}\left(  \sigma_{k}\mathcal{L}_{u}\left[  \partial_{x}U_{k}^{\ast}%
			+\partial_{x}\omega_{k}\right]  +\tau_{k}\mathcal{L}_{u}\left[  \partial
			_{y}U_{k}^{\ast}+\partial_{y}\omega_{k}\right]  \right) \\
			&  +\lambda\left(  \phi^{\ast}+\sum\limits_{k=1}^{3}\left(  \sigma_{k}%
			\partial_{x}\omega_{k}+\tau_{k}\partial_{y}\omega_{k}\right)  \right)  .
		\end{align*}
		Let us denote the right hand side by $Q.$
		
		For a function $\eta,$ we define
		\begin{align*}
			\eta_{k,x}^{\parallel}  =\int_{\mathbb{R}^{2}}\eta\partial_{x}U_{k}^{\ast
			}dxdy,\quad \eta_{k,y}^{\parallel}=\int_{\mathbb{R}^{2}}\eta\partial
			_{y}U_{k}^{\ast}dxdy,
		\end{align*}
		and
		\begin{align*}
			\eta^{\Vert}  =\sum\limits_{k=1}^{3}\left(  \eta_{k,x}^{\parallel}%
			\partial_{x}U_{k}^{\ast}+\eta_{k,y}^{\parallel}\partial_{y}U_{k}^{\ast
			}\right)  ,
		\end{align*}
		and $\eta^{\perp}=\eta-\eta^{\Vert}.$ The function $\eta^{\perp}$ can be
		understood as the projection orthogonal to the kernels of the linearized
		operator. Then
		\begin{equation}
			\mathcal{L}_{u}\zeta^{\ast}-\lambda\xi^{\ast}=Q^{\perp}+Q^{\Vert}.
			\label{kistar}%
		\end{equation}

		Estimates $\left(  \ref{e1}\right)  $ and $\left(  \ref{e11}\right)  $ imply
		that $Q^{\perp}=O\left(  \gamma^{-3}\right)  .$ On the other hand, multiplying
		$\left(  \ref{kistar}\right)  $ by $\partial_{x}U_{k}^{\ast},\partial_{y}%
		U_{k}^{\ast}$ and integrating tell us that $Q^{\Vert}=o\left(  \zeta^{\ast
		}\right)  .$ Following the proof of Lemma \ref{apries}, we find that there
		exists $\varepsilon>0$, such that
		\[
		\left\Vert \zeta^{\ast}\theta_{\varepsilon}\right\Vert _{L^{\infty}\left(
			\mathbb{R}^{2}\right)  }\leq C\gamma^{-3}.
		\]

		With the estimate of $\zeta$ at hand, now we project the equation
		$\mathcal{L}_{u}\phi=\lambda\phi$ onto the kernels $\partial_{x}U_{k}^{\ast}$
		and $\partial_{y}U_{k}^{\ast}.$ More precisely, we consider the following two
		identities
		\begin{align*}
			\int_{\mathbb{R}^{2}}\partial_{x}U_{k}^{\ast}\left(  \mathcal{L}_{u}%
			\phi-\lambda\phi\right)  dxdy =0,\quad
			\int_{\mathbb{R}^{2}}\partial_{y}U_{k}^{\ast}\left(  \mathcal{L}_{u}%
			\phi-\lambda\phi\right)  dxdy =0.
		\end{align*}
		Now we compute
		\begin{align*}
			\int_{\mathbb{R}^{2}}\partial_{x}U_{j}^{\ast}\mathcal{L}_{u}\left[
			\partial_{x}\omega_{j}\right]  dxdy
			&  =6\int_{\mathbb{R}^{2}}\partial_{x}U_{j}^{\ast}\left( \sum_{k\neq j} U_{k}^{\ast}%
			+\xi\right)  \partial_{x}\omega_{j}dxdy\\
			&  =6\int_{\mathbb{R}^{2}}\left(  -\omega_{j}\mathbf{p}_{j}-\frac{1}{2}%
			\omega_{j}^{2}\right)  \partial_{x}^{2}U_{j}^{\ast}dxdy+O\left(  \gamma
			^{-5}\right) \\
			&  =-a^{\ast}\mathbf{p}_{j}^{2}+O\left(  \gamma^{-5}\right)  .
		\end{align*}
		Similarly, we have%
		\begin{align*}
			\int_{\mathbb{R}^{2}}\partial_{x}U_{j}^{\ast}\mathcal{L}_{u}\left[
			\partial_{y}\omega_{j}\right]  dxdy
			&  =6\int_{\mathbb{R}^{2}}\partial_{x}U_{j}^{\ast}\left(  \sum_{k\neq j}
			U_{k}^{\ast}+\xi\right)  \partial_{y}\omega_{j}dxdy\\
			&  =6\int_{\mathbb{R}^{2}}\left(  -\omega_{j}\mathbf{p}_{j}-\frac{1}{2}%
			\omega_{j}^{2}\right)  \partial_{x}\partial_{y}U_{j}^{\ast}dxdy\\
			&  =-b^{\ast}\mathbf{p}_{j}^{2}+O\left(  \gamma^{-5}\right),
		\end{align*}
		and
		\begin{align*}
			\int_{\mathbb{R}^{2}}\partial_{y}U_{j}^{\ast}\mathcal{L}_{u}\left[
			\partial_{y}\omega_{j}\right]  dxdy
			&  =6\int_{\mathbb{R}^{2}}\partial_{y}U_{j}^{\ast}\left(  \sum_{k\neq j}
			U_{k}^{\ast}+\xi\right)  \partial_{y}\omega_{j}dxdy\\
			&  =6\int_{\mathbb{R}^{2}}\left(  -\omega_{j}\mathbf{p}_{j}-\frac{1}{2}%
			\omega_{j}^{2}\right)  \partial_{y}^2U_{j}^{\ast}dxdy\\
			&  =-c^{\ast}\mathbf{p}_{j}^{2}+O\left(  \gamma^{-5}\right).	
		\end{align*}
		While for $j\neq k$, we have
		\begin{align*}
			\int_{\mathbb{R}^{2}}\partial_{x}U_{j}^{\ast}\mathcal{L}_{u}\left[
			\partial_{x}\omega_{k}\right]  dxdy
			=O\left(  \gamma^{-5}\right),\quad
			\int_{\mathbb{R}^{2}}\partial_{x}U_{j}^{\ast}\mathcal{L}_{u}\left[
			\partial_{y}\omega_{k}\right]  dxdy=O\left(  \gamma^{-5}\right),
		\end{align*}
		and
		\begin{align*}
			\int_{\mathbb{R}^{2}}\partial_{y}U_{j}^{\ast}\mathcal{L}_{u}\left[
			\partial_{y}\omega_{k}\right]  dxdy =O\left(  \gamma^{-5}\right).
		\end{align*}
		From the above estimates and Lemma \ref{pro1}, we deduce that for $k=1,2,3,$%
		\[
		\int_{\mathbb{R}^{2}}\left(  \partial_{x}U_{k}^{\ast}-i\partial_{y}U_{k}%
		^{\ast}\right)  \left(  \mathcal{L}_{u}\phi-\lambda\phi\right)  dxdy=\sum
		\limits_{j=1}^{3}\left(  F_{k,j}\mathbf{e}_{j}\right)  +O\left(  \gamma
		^{-5}\right)  ,
		\]
		where $\mathbf{e}_{j}=\sigma_{j}+i\tau_{j}$ is complex number. Moreover, there
		exists universal positive constants $\delta_{1},\delta_{2},$ such that
		\[
		\delta_{1}\leq\sum\limits_{j}\left\vert \mathbf{e}_{j}\right\vert ^{2}%
		\leq\delta_{2}.
		\]
		Hence, after a scaling, in terms of the matrix $M$ defined in Lemma
		\ref{vector}, we obtain%
		\[
		M\left(  \mathbf{e}_{1},\mathbf{e}_{2},\mathbf{e}_{3}\right)  ^{T}=O\left(
		\gamma^{-1}\right)  +O\left(  \left\vert \lambda\right\vert \right)  .
		\]
		We then deduce that for some constants $c_{1},c_{2},$ with $c_{1}^{2}%
		+c_{2}^{2}$ uniformly bounded away from zero, such that
		\[
		\left(  \mathbf{e}_{1},\mathbf{e}_{2},\mathbf{e}_{3}\right)  =c_{1}%
		\mathbf{b}_{1}+c_{2}\mathbf{b}_{2}+O\left(  \gamma^{-1}\right)  +O\left(
		\left\vert \lambda\right\vert \right)  .
		\]
		On the other hand, using the explicit four-parameter family of solutions
		$u_{A,B}$ for the Boussinesq equation, we deduce that there exists a kernel
		$\varphi$ of $\mathcal{L}_{u}$ whose projection on $\partial_{x}U_{k}^{\ast
		},\partial_{y}U_{k}^{\ast}$ is close to $c_{1}\mathbf{b}_{1}+c_{2}%
		\mathbf{b}_{2}.$ We then conclude that $\phi_{j}$ will not be $L^{2}%
		$-orthogonal to the $\varphi$ for $j$ large enough. This contradicts with the
		fact that $\phi_{j}$, as an eigenfunction with respect to a negative
		eigenvalue, has to be $L^{2}$-orthogonal to $\varphi.$
		
		Now we have proved that negative eigenvalue $\lambda_{j}$ will be uniformly
		bounded away from $0.$ Using this information, we then deduce that
		$\lambda_{j}$ has to converge to the unique negative eigenvalue of the
		operator $\mathcal{L}_{U}.$\ By result in \cite{Liu}, the Morse index of the
		standard lump $U$ is equal to $1,$ we then find that the Morse index of $u$ is
		at most $3.$ On the other hand, by constructing explicit test functions using
		the eigenfunctions(of negative eigenvalue) of $\mathcal{L}_{U}$, we know that
		the Morse index of $u$ is at least $3.$ Hence we conclude that the Morse index
		of $u$ is equal to $3.$ This finishes the proof.
	\end{proof}
	
	Having analyzed the solutions $u_{0,B}$ for $B$ large, we proceed to show that
	all the solutions $u_{A,B}=2\partial_{x}^{2}\ln h_{A,B}$ has Morse index $3.$
	
	\begin{proof}
		[Proof of Theorem \ref{main}]Straightforward application of the results in the
		previous section tells us that there is a Backlund transformation from the
		translated degree $2$ polynomial
		\[
		f\left(  x,y\right)  =x^{2}+\left(  y-\frac{2\sqrt{3}i}{3}\right)  ^{2}+3
		\]
		to the degree $4$ polynomial $g,$ which is defined explicitly by
		\begin{align*}
			g\left(  x,y\right)     =~&x^{4}+2ix^{3}y+2ixy^{3}-y^{4}+\frac{10x^{3}}%
			{\sqrt{3}}+4\sqrt{3}ix^{2}y+2\sqrt{3}xy^{2}+\frac{\sqrt{3}iy^{3}}{8}\\
			&  +20x^{2}+30ixy+2y^{2}+\left(  \frac{50}{\sqrt{3}}+\frac{Ai+B}{2}\right)
			x\\
			&  +\left(  \frac{80i}{\sqrt{3}}+\frac{A-Bi}{2}\right)  y-25+\frac
			{Ai+B}{2\sqrt{3}}.
		\end{align*}
		Then the function $g$ is Backlund transformed to $h_{A,B}.$ Observe that both
		$f$ and $g$ have finitely many simple zeros. Using this fact, the
		nondegeneracy of the linearized operator $\mathcal{L}_{u_{A,B}}$ then follows
		directly from the same argument as that of \cite{Liu}, that is, by analyzing
		the associated linearized Backlund transformations.
		
		\medskip
		
		To show that the Morse index of $u_{A,B}$ is equal to $3,$ let $m$ be
		sufficiently large and consider the family of solutions $u_{At,Bt+\left(
			1-t\right)  m}.$ When $t=1,$ it is $u_{A,B},$ and for $t=0,$ the solution is
		$u_{0,m}$. We now know that for any $t\in\left[  0,1\right]  ,$ the solution
		is nondegenerated in the sense that the linearized operator has no nontrivial
		kernels. This together with the continuous dependence of the negative
		eigenvalues with respect to $t$, implies that as $t$ decreases from $1$ to
		$0,$ a negative eigenvalue can not diminish to the zero eigenvalue. We then
		see that the Morse index of $u_{A,B}$ should be the same as that of $u_{0,m}.$
		By Proposition \ref{Morse}, $u_{0,m}$ has Morse index $3$, provided that $m$
		is large. We then conclude that for any $A,B,$ the Morse index of $u_{A,B}$
		equals $3.$ This completes the proof.
	\end{proof}
	
	As a final remark, we point out that for solutions with higher degrees, the
	above arguments also work. The only delicate part is, we need to show, as one
	of the parameter tends to infinity, the solution splits into a number of
	classical lumps which are far away from each other. This is to ensure that
	procedure of reverse Lyapunov-Schmidt reduction can be started. For polynomial
	tau functions of degree $k\left(  k+1\right)  ,$ the corresponding solution
	should have Morse index $k\left(  k+1\right)  /2.$ Once this is proved, it will yield the existence of infinitely many solutions
	for the GP and generalized KP equation. Rigorous
	justification of this fact would require a complete classification of the moduli space of lump type solutions. This is an ongoing project.

\end{document}